\documentclass[12pt,a4paper]{article}

\usepackage[english]{babel}
\usepackage[utf8]{inputenc}
\usepackage[T1]{fontenc}

\usepackage[style=numeric]{biblatex}
\bibliography{bibliography}
\renewbibmacro{in:}{}
\usepackage[right=3cm, a4paper]{geometry}
\usepackage{csquotes}
\usepackage{amsmath}
\usepackage{amssymb}	
\usepackage{amsthm}
\usepackage{mathtools}
\usepackage{tikz-cd}
\usepackage{float}
\usepackage{graphicx}

\usepackage[colorinlistoftodos, backgroundcolor=yellow!10, linecolor=red!60, textwidth=2.5cm, textsize=footnotesize]{todonotes}

\usepackage[colorlinks=true, allcolors=blue]{hyperref}
\usepackage{enumitem}
\usepackage{svg}
\usepackage{scalerel}

\usepackage{multicol}
\usepackage{framed}
\usepackage[font=small, labelfont=it, textfont=it]{caption}
\usepackage[font=small]{subcaption}

\newcommand{\fld}{\ensuremath{\Bbbk}}

\newcommand{\Vertices}{\operatorname{Vert}}
\newcommand{\Edges}{\operatorname{Edges}}

\newcommand{\BV}{\Delta}
\newcommand{\cp}{\square}

\newcommand{\hdeg}[1]{{|#1|}}

\newcommand{\ot}{\otimes}

\newcommand{\g}{\ensuremath{\mathfrak{g}}}
\newcommand{\h}{\ensuremath{\mathfrak{h}}}
\newcommand{\q}{\ensuremath{\mathfrak{q}}}

\newcommand{\ad}{\ensuremath{\operatorname{ad}}}
\newcommand{\Ad}{\ensuremath{\operatorname{Ad}}}
\newcommand{\Sym}{\ensuremath{\operatorname{Sym}}}
\newcommand{\otr}{\ensuremath{\operatorname{otr}}}
\newcommand{\odet}{\ensuremath{\operatorname{odet}}}
\newcommand{\tr}{\ensuremath{\operatorname{Tr}}}

\newcommand{\Hom}{{\ensuremath{\mathrm{Hom}}}}

\newcommand{\hol}{{\ensuremath{\mathrm{hol}}}}

\newcommand{\egm}{\Phi^\textnormal{even}}
\newcommand{\ogm}{\Phi^\textnormal{odd}}
\newcommand{\egmex}{\tilde{\Phi}^\textnormal{even}}
\newcommand{\ogmex}{\tilde{\Phi}^\textnormal{odd}}

\newcommand{\Goldman}[1]{\mathcal G_{#1}}
\newcommand{\Goldmanred}[1]{\mathcal G^\mathrm{red}_{#1}}

\newcommand{\ModSpace}[2]{\mathcal{M}_{#1}(#2)}

\newcommand{\he}{\mathrm{he}}

\newtheorem{theorem}{Theorem}

\newtheorem{proposition}{Proposition}
\newtheorem{corollary}{Corollary}
\newtheorem{conjecture}{Conjecture}

\newtheorem{definition}{Definition}
\newtheorem{remark}{Remark}
\newtheorem{example}{Example}

\begin{document}

\title{Batalin-Vilkovisky structures on moduli spaces of flat connections}
\date{}
\author{Anton Alekseev\thanks{Section of Mathematics, University of Geneva, Rue du Conseil-Général 7-9, 1205 Geneva, Switzerland, \texttt{anton.alekseev@unige.ch}}, Florian Naef\thanks{School of Mathematics, Trinity College, Dublin 2, Ireland, \texttt{naeff@tcd.ie}}, Ján Pulmann\thanks{School of Mathematics, University of Edinburgh, Peter Guthrie Tait Road, Edinburgh, U.K., \texttt{Jan.Pulmann@ed.ac.uk}}, Pavol \v{S}evera\thanks{Section of Mathematics, University of Geneva, Rue du Conseil-Général 7-9, 1205 Geneva, Switzerland, \texttt{pavol.severa@gmail.com}}}

\maketitle
\abstract{
Let $\Sigma$ be a compact oriented 2-manifold (possibly with boundary), and let $\Goldman{\Sigma}$ be the linear span of free homotopy classes of closed oriented curves on $\Sigma$ equipped with the Goldman Lie bracket $[\cdot, \cdot]_\text{Goldman}$ defined in terms of intersections of curves. A theorem of Goldman gives rise to a Lie homomorphism $\Phi^{\rm even}$ from
$(\Goldman{\Sigma}, [\cdot, \cdot]_\text{Goldman})$ to functions on the moduli space of flat connections $\ModSpace{\Sigma}{G}$ for $G=U(N),GL (N)$, equipped with the Atiyah-Bott Poisson bracket. 

The space $\Goldman{\Sigma}$ also carries the Turaev Lie cobracket $\delta_\text{Turaev}$ defined in terms of self-intersections of curves. 
In this paper, we address the following natural question: which geometric structure on moduli spaces of flat connections corresponds to the Turaev cobracket?

We give a constructive answer to this question in the following context: 
for $G$ a Lie supergroup with an odd invariant scalar product on its Lie 
superalgebra, and for nonempty $\partial\Sigma$, we show that the moduli space of flat connections $\ModSpace{\Sigma}{G}$ 
carries a natural Batalin-Vilkovisky (BV) structure, given by an explicit 
combinatorial Fock-Rosly formula. Furthermore, for the queer Lie supergroup 
$G=Q(N)$, we define a BV-morphism $\ogm\colon \wedge \Goldman{\Sigma} \to {\rm Fun}(\ModSpace{\Sigma}{Q(N)})$ 
which replaces the Goldman map, and which captures the information both on the 
Goldman bracket and on the Turaev cobracket. The map $\ogm$ is constructed
using the ``odd trace'' function on $Q(N)$.
}

\section{Introduction}

Let $G$ be a connected Lie group, $\g$ its Lie algebra, and $\Sigma$ a connected oriented 2-dimensional manifold (possibly with boundary). We denote by
$$
\ModSpace{\Sigma}{G}=\mathrm{Hom}(\pi_1(\Sigma), G)/G
$$
the moduli space of flat connections on $\Sigma$. Usually, the action of $G$ on the representation space $\mathrm{Hom}(\pi_1(\Sigma), G)$ is not free, and it doesn't need to be proper. For this reason, the space $\ModSpace{\Sigma}{G}$ is often singular, and there are various approaches to dealing with it. For the purpose of this introduction, we will be considering the smooth part of the moduli space.

Assume that $\g$ admits an invariant scalar product. Then, by the Atiyah-Bott Theorem \cite{AtiyahBott}, $\ModSpace{\Sigma}{G}$ carries a canonical Poisson structure $\{ \cdot, \cdot\}_\text{Atiyah-Bott}$. If $\Sigma$ is closed, this Poisson structure is non-degenerate, and $\ModSpace{\Sigma}{G}$ becomes a symplectic space.

In \cite{Goldman}, Goldman discovered an extraodinary relation between the Atiyah-Bott Poisson structure on $\ModSpace{\Sigma}{G}$ and the topology of oriented closed curves on $\Sigma$. 
In more detail, let $\gamma_1, \gamma_2$ be two such curves on $\Sigma$ which intersect transversally at a finite number of points. Denote their free homotopy classes by $|\gamma_1|, |\gamma_2|$, and define the Goldman bracket by the formula
\begin{equation}
    [|\gamma_1|, |\gamma_2|]_\text{Goldman} = \sum_{p \in \gamma_1 \cap \gamma_2} \epsilon_p 
\, |\gamma_1 \cdot_p \gamma_2|.
\end{equation}
Here  $\epsilon_p$ is the sign of the intersection (with respect to the orientation of $\Sigma$) and $\gamma_1 \cdot_p \gamma_2$ is the oriented connected sum of $\gamma_1$ and $\gamma_2$ at $p$.
\begin{theorem}[Goldman \cite{Goldman}]  \label{intro:Goldman}
The bracket $[\cdot , \cdot]_\textnormal{Goldman}$ is well defined, and it gives rise to a Lie bracket on the linear span $\Goldman{\Sigma}$ of free homotopy classes of oriented closed curves on $\Sigma$.
\end{theorem}

Furthermore, let $G=U(n)$ or $GL(n)$,  and choose $\gamma \in \pi_1(\Sigma)$. The Goldman function on $\ModSpace{\Sigma}{G}$, associated to $\gamma$, is defined by
$$
\egm_{|\gamma|}([\rho]) = \tr \, \rho(\gamma),
$$
where $\rho \in {\rm Hom}(\pi_1(\Sigma), G)$ is a lift of $[\rho] \in \ModSpace{\Sigma}{G}$. The function $\egm_{|\gamma|}$ is well defined, and it depends only on the free homotopy class of $\gamma$. The following theorem establishes a relation between Goldman brackets and Atiyah-Bott Poisson structures:
\begin{theorem}[Goldman \cite{Goldman}]     \label{intro:AB=Goldman}
$$
\{ \egm_{|\gamma_1|}, \egm_{|\gamma_2|}\}_\textnormal{Atiyah-Bott} = \egm_{[|\gamma_1|, |\gamma_2|]_\textnormal{Goldman}}.
$$
\end{theorem}
Theorem \ref{intro:AB=Goldman} can be restated in several ways. First, it says that the map
$$
\egm\colon (\Goldman{\Sigma}, [ \cdot, \cdot]_\text{Goldman}) \to ({\rm Fun}(\ModSpace{\Sigma}{G}),
\{ \cdot, \cdot\}_\textnormal{Atiyah-Bott})
$$
is a Lie algebra homomorphism. Equivalently, the symmetric algebra $S\Goldman{\Sigma}$ is naturally a Poisson algebra with {$[\cdot, \cdot]_\text{Goldman}$ extended to a Poisson bracket}, and the map $\egm$ extends to a homomorphism of Poisson algebras
$$
\egm\colon S\Goldman{\Sigma} \to {\rm Fun}(\ModSpace{\Sigma}{G}).
$$
If one considers polynomial functions on the moduli space, this map is actually surjective for all $n$, see \cite{Procesi1976} (recall that  $G=U(n)$ or $GL(n)$).

Note that the class of a trivial loop $\bigcirc \in \Goldman{\Sigma}$ belongs to the center of 
the Lie algebra $\Goldman{\Sigma}$, and that the Goldman bracket descends to the quotient space
$$
\Goldmanred{\Sigma}=\Goldman{\Sigma}/\mathbb{R}\bigcirc.
$$
In \cite{Turaev}, Turaev showed that the space $\Goldman{\Sigma}$ carries a natural Lie cobracket
$$
\delta_{\rm Turaev}(|\gamma|) = \sum_{p \in \gamma \cap \gamma} \epsilon_p \, |\gamma'_p| \wedge |\gamma''_p|.
$$
Here we assume that the curve $\gamma$ has a finite number of transverse self-intersections which are denoted by $p \in \gamma \cap \gamma$. By resolving the oriented intersection at the point $p$, we obtain two closed oriented curves $\gamma'_p$ and $\gamma''_p$. Similar to the definition of the Goldman bracket, $\epsilon_p$ is the sign of the intersection with respect to the orientation of $\Sigma$.
\begin{theorem}[Turaev \cite{Turaev}, Chas \cite{Chas2004}]
The triple $(\Goldmanred{\Sigma}, [\cdot, \cdot]_{\rm Goldman}, \delta_{\rm Turaev})$ is an involutive Lie bialgebra. That is, $\delta_{\rm Turaev}$ is well defined, it is a Lie cobracket and a Lie algebra 1-cocycle with respect to the Goldman bracket, and the composition map
$
[\cdot, \cdot]_{\rm Goldman} \circ \delta_{\rm Turaev}
$
vanishes.
\end{theorem}
In this paper, we address the following natural question: 

\vskip 0.2cm
{\bf Question}: what is the geometric structure on the moduli space of flat connections which corresponds to the Turaev cobracket?

\vskip 0.2cm

We start with the following standard wisdom, see e.g. \cite[Sec.~5]{CMW2016}:
\begin{theorem}\label{thm:CEintro}
Let $(\mathcal{G}, [\cdot, \cdot], \delta)$ be an involutive Lie bialgebra. Then, the exterior algebra $\wedge \mathcal{G}$ is naturally a Batalin-Vilkovisky (BV) algebra. That is, it carries a unique second order differential operator $\Delta \equiv \Delta^{[\cdot, \cdot], \delta}$ such that 
$$
\Delta^2=0, \hskip 0.3cm \Delta(1)=0, \hskip 0.3cm \Delta(x)=\delta(x),
\hskip 0.3cm
\Delta(x\wedge y)=\Delta(x) \wedge y - x\wedge \Delta(y) + [x,y]
$$
for all $x,y \in \mathcal{G}$.
\end{theorem}
Our first result provides a refinement of the Atiyah-Bott Poisson structures in the case when a Lie group  is replaced by a Lie supergroup, and an invariant scalar product on the Lie algebra is replaced by an odd invariant scalar product on the Lie superalgebra.  In contrast to the even case, it is in general necessary to consider foliated surfaces, i.e. surfaces with a decomposition into 1-dimensional submanifolds, as on Figure \ref{fig:folintro}.

\begin{figure}[h]
    \centering
    \includesvg{images/folintro.svg}
    \caption{Foliated annulus and once-punctured torus. Note that we don't require the foliation to be oriented nor compatible with boundary. See Section \ref{ssec:foliation} and Appendix \ref{app:foliation} for more details.}  
    \label{fig:folintro}
\end{figure}
\begin{theorem}  \label{intro:BV_moduli}
Let $G$ be a Lie supergroup with an odd invariant scalar product on its Lie superalgebra, and let $\Sigma$ be an oriented 2-dimensional manifold with nonempty boundary, equipped with a foliation $f$. Then, the moduli space $\ModSpace{\Sigma}{G}$ carries a  natural BV structure $\Delta^f$. Futhermore, if $\g={\rm Lie}(G)$ is a unimodular Lie superalgebra, then  $\Delta^f$ is independent of $f$.
\end{theorem}

Next, we consider the case of the Lie supergroup $G=Q(n)$. Its Lie superalgebra $\q(n)$ is unimodular and carries an odd invariant scalar product.  Hence, moduli spaces $\ModSpace{\Sigma}{Q(n)}$ carry canonical BV structures
$\Delta$. Elements of $Q(n)$ are expressions of the form
$$
u=u_{\rm even} + \xi \, u_{\rm odd},
$$
where $u_{\rm even}, u_{\rm odd} \in \operatorname{Mat}_{n}(\mathbb R)$, and $\xi$ is an odd element with $\xi^2=1$. We define an analogue of the Goldman map
$$
\ogm\colon |\gamma| \mapsto \ogm_{|\gamma|}([\rho]) = {\rm otr}(\rho({\gamma})),
$$
where ${\rm otr}(u)=\tr(u_{\rm odd})$. Images of elements $|\gamma| \in \Goldman{\Sigma}$ are odd functions on the moduli space, and we extend the map $\ogm$ to
$$
\ogm\colon \wedge \Goldmanred{\Sigma} \to {\rm Fun}(\ModSpace{\Sigma}{Q(n)}).
$$

Our second result is given by the following theorem:
\begin{theorem}\label{thm:OddGoldman}
The map $\ogm$ is a morphism of BV algebras with respect to the BV structure on $\wedge \Goldman{\Sigma}$ defined by the Goldman-Turaev Lie bialgebra $\Goldmanred{\Sigma}$, and the BV structure on ${\rm Fun}(\ModSpace{\Sigma}{Q(n)})$ defined by Theorem \ref{intro:BV_moduli}.
\end{theorem}

The supergroup $Q(n)$ comes with another odd function $\odet$, which leads to an extension of the Goldman-Turaev Lie bialgebra to $\Goldmanred{\Sigma}\oplus H_1(\Sigma)$. The bracket between $\Goldmanred{\Sigma}$ and $H_1(\Sigma)$ is given by the intersection pairing, and the cobracket is extended by $0$ to $H_1(\Sigma)$. An extension of the map $\ogm$, sending a 1-cycle $[a]$ to the function  $$\rho \mapsto \odet \rho(a),$$ is a well defined  morphism of BV algebras.

Theorem \ref{thm:OddGoldman} is an analogue of the Goldman's result for $G=U(N), {\rm GL}(N)$, and it gives a constructive answer to the main question addressed in the paper.

In order to prove these results, we use the technique introduced by Fock-Rosly \cite{FockRosly}. In more detail, we choose a finite set $V \subset \partial \Sigma$ and consider a subgroupoid $\Pi_1(\Sigma, V)$ with base points $V$. The set of representations
$$
\ModSpace{\Sigma, V}{G}={\rm Hom}(\Pi_1(\Sigma, V), G),
$$
has the property that\footnote{By $G^V$, we mean a collection of elements of $G$ indexed by $V$, i.e. a morphism of sets $V\to G$.}
$$
\ModSpace{\Sigma}{G} \cong \ModSpace{\Sigma, V}{G}/G^V.
$$
In contrast to $\ModSpace{\Sigma}{G}$, the space $\ModSpace{\Sigma, V}{G}$ admits easy descriptions given a choice of an embedded graph $\Gamma \subset \Sigma$ with the following properties: the set of vertices of $\Gamma$ is $V$, and the surface $\Sigma$ retracts to $\Gamma$. 
Then, $\ModSpace{\Sigma, V}{G} \cong G^E$, where $E$ is the set of edges of $\Gamma$.

In this context, Fock and Rosly \cite{FockRosly} defined a bivector $\pi^\Gamma_\text{Fock-Rosly}$ on $\ModSpace{\Sigma, V}{G}$ given by an explicit combinatorial formula. In more detail, let $s \in \g \otimes \g$ be the canonical element with respect to the invariant scalar product on $\g$. Then,
\begin{equation}       \label{intro:FR}
\pi^\Gamma_{\text{Fock-Rosly}} = \frac{1}{2} \, \sum_{v \in V} \sum_{a<b \in S(v)} s_{ab}.
\end{equation}
Here $S(v)$ is the star of the vertex $v$ which consists of half-edges with endpoint $v$, $a<b$ refers to the order of elements of $S(v)$ induced by the orientation of $\Sigma$, and $s_{ab}$ is a bidifferential operator  acting on pairs of functions on $G^E$ by differentiating the copies of $G$ corresponding to the half-edges $a$ and $b$. The following theorem summarises several results:
\begin{theorem}[Fock-Rosly \cite{FockRosly}, Massuyeau-Turaev \cite{MassuyeauTuraev2012}, Nie \cite{Nie2013}, Li-Bland-\v{S}evera \cite{LBSQuilted}]
Bivectors defined on $\ModSpace{\Sigma, V}{G}$ by different choices of the graph $\Gamma$ coincide with each other. The bivector $\pi_\textnormal{Fock-Rosly}$ descends to $\ModSpace{\Sigma}{G}$ giving $\{ \cdot, \cdot\}_\textnormal{Atiyah-Bott}$. 
\end{theorem}
It is worth noting that $\pi_\textnormal{Fock-Rosly}$ is not a Poisson bivector. Instead, it has a controllable defect in the Jacobi identity. Such bivectors are called quasi-Poisson bivectors \cite{AKS, AKSM}, and they have many properties similar to Poisson bivectors.

In the case when $G$ is a Lie supergroup with an odd invariant scalar product on its Lie superalgebra $\g$, we introduce a canonical element $t \in \g \otimes \g$. Assuming that $\g$ is unimodular, we define a second order differential operator on $\ModSpace{\Sigma, V}{G} \cong G^E$:
\begin{equation}       \label{intro:Delta}
\Delta^\Gamma = \frac{1}{2} \, \sum_{v \in V} \sum_{a<b \in S(v)} t_{ab}.
\end{equation}
In comparison to equation \eqref{intro:FR}, there are two differences: first, $t$ is an odd element while $s$ is even. Second, we view $s_{ab}$ as a bidifferential operator acting on a pair of functions on $G^E$ while we consider $t_{ab}$ as a second order differential operator acting on one function. Despite these differences, properties of $\Delta^\Gamma$ resemble those of the Fock-Rosly bivector:
\begin{theorem}  \label{intro:quasiBV}
Second order differential operators defined on $\ModSpace{\Sigma, V}{G}$ by different choices of the graph $\Gamma$ coincide with each other. The operator $\Delta$ descends to a BV operator on $\ModSpace{\Sigma}{G}$.
\end{theorem}
Again, the operator $\Delta$ is not a BV operator. Instead, its square is non-vanishing in a controllable way giving rise to a notion of quasi-BV operators. To simplify the presentation, Theorem \ref{intro:quasiBV} is stated for the case of unimodular Lie superalgebras. A more general statement can be found in the body of the paper.

The structure of the paper is as follows: in Section 2, we recall the Fock-Rosly construction. In Section 3, we describe an analogous construction of the quasi-BV operator $\Delta$ for supergroups. In Section 4, we give an alternative, topological construction of $\Delta$, via intersections of curves. In Section 5, we consider moduli spaces for the Lie supergroup $Q(n)$ and we establish a relation between $\Delta$ and the Goldman-Turaev Lie bialgebra. Appendix A is devoted to skeletons and foliations on surfaces with boundary. Appendix B contains information on Lie algebras and Lie superalgebras with an (even or odd) invariant scalar product. Appendix C provides a Hopf algebra viewpoint on the fusion operation for quasi-BV spaces.

\subsection*{Acknowledgements}
We are indebted to V. Serganova for fruitful discussions, 
and we would like to thank E. Getzler and V. Turaev for their interest in our work. We would also like to thank the referee for a careful reading of the manuscript, and Aoi Wakuda for noticing a sign ambiguity in Appendix \ref{app:oddLie}.

Research of A.A. was supported in part by the grants 208235 and 200400 and by the National Center for Competence in Research (NCCR) SwissMAP of the Swiss National Science Foundation, and by the award of the Simons Foundation to the Hamilton Mathematics Institute of the Trinity College Dublin under the program “Targeted Grants to Institutes”.
Research of F.N. was supported in part by the Marie Sklodowska-Curie action, grant agreement no. 896370.
Research of J.P. was supported by the Postdoc.Mobility grant 203065 of the SNSF.

\clearpage

\section{Quasi-Poisson structures on moduli spaces}
In this section, we recall a description of the Poisson structure on the moduli space of flat connections on a surface, due to Fock and Rosly \cite{FockRosly}. We start with some useful results on surfaces. 

\subsection{Surfaces}
\subsubsection{Skeletons}
Let $\Sigma$ be an oriented, compact surface with boundary and $\{p_1, \dots, p_n \} = V\subset \partial \Sigma${}
a finite, non-empty subset of \emph{marked points}.
Recall that  $\Pi_{1}(\Sigma, V)$, the fundamental groupoid with base $V$, is 
 the full subgroupoid of the fundamental groupoid $\Pi_{1}(\Sigma)$ on objects $V\subset\Sigma$.  In other words,
 the set of objects of $\Pi_{1}(\Sigma, V)$ is $V$, while morphisms $p_i \to p_j$ are homotopy classes of paths from $p_i$ to $p_j$.
 
  By a graph, we mean a 1-dimensional CW complex $\Gamma$. The sets of vertices and edges of $\Gamma$ are denoted $\Vertices{\Gamma}$ and $\Edges{\Gamma}$. For each
 vertex $p\in \Vertices{\Gamma}$, we denote by $\he{(p)}$ the set of half-edges of $p$.
\begin{definition}
A \emph{skeleton} of $\Sigma$ is a topological embedding $\Gamma\hookrightarrow \Sigma$ of an oriented graph $\Gamma$ such that
\begin{enumerate}
	\item  restricted to each edge $e \in \Edges\Gamma$, the injection $e\hookrightarrow \Sigma$ is a smooth embedding of manifolds with boundary,
	\item the image of $\Vertices\Gamma$ equals $V$, with no other intersection of the image of $\Gamma$ with $\partial \Sigma$,
	\item $\Sigma$ deformation retracts to the image of $\Gamma$. 
\end{enumerate}
\end{definition}
The edges of such skeleton then freely generate the fundamental groupoid  $\Pi_{1}(\Sigma, V)$, under
composition and inversion.
\begin{remark}
	Skeletons are closely related to ciliated graphs of \cite{FockRosly}.
	Indeed, to any skeleton we can associate a ciliated graph using the orientation
	of the surface, with the cilium	pointing outside of the surface. The thickening 
	of such graph is then homeomorphic to the original surface.
\end{remark}
\begin{remark}
	An alternative approach to describing the fundamental groupoid $\Pi_{1}(\Sigma, V)$
	and thus the moduli space of flat connections, would be to use triangulations,
	or their dual uni-trivalent graphs. This approach 
	would result in slightly more complicated computations in Sections \ref{ssec:FR} and \ref{sec:qBVFR},
	so we have chosen to use skeletons.
\end{remark}

It is easy to see that one can always find a skeleton. Moreover, any two skeletons can be connected
by isotopy, edge reversions and \emph{slides}. A slide is a move of a half-edge along a neighboring edge -- see Figure \ref{fig:moves}.

\begin{proposition}\label{prop:2Dtopology}
	Any two skeletons $\Gamma, \Gamma' \subset \Sigma$ are connected by a finite sequence of isotopies, 
	edge  reversions and slides.
\end{proposition}

\begin{figure}[h!]
	\centering
	\begin{subfigure}[c]{0.3\textwidth}
		\begin{center}
			\includesvg{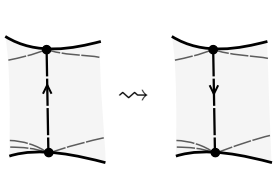}
			\vspace*{8mm}
		\end{center}
		\caption{Reversing an edge.}
		\label{fig:convention}
	\end{subfigure}
	\qquad  
	\begin{subfigure}[c]{0.52\textwidth}
		\begin{center}
			\includesvg[scale=0.7]{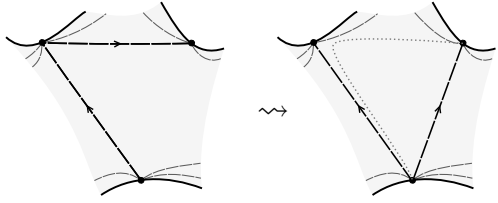}
		\end{center}
		\caption{Sliding the top edge along the left edge. Note that some of the marked 
			vertices could be identified.}
		\label{fig:slide}
	\end{subfigure}
	\caption{The moves between skeletons. The solid line is the boundary $\partial \Sigma$.}\label{fig:moves}
\end{figure}

This fact seems to be known among experts, see the work of Bene \cite[Theorem~5.3]{Bene2010} (where
$\Sigma$ has one boundary component and one marked point) and Jackson \cite[Corollary~6.21]{JacksonThesis} 
(where each boundary component has one marked point). We present a 
proof of this version of the claim in Appendix~\ref{app:skeleton}.

\medskip

\subsubsection{Foliations}
\label{ssec:foliation}
For Batalin-Vilkovisky structures on moduli spaces, we will need surfaces
equipped with a 1-dimensional foliation. Since $\Sigma$ has a boundary, we will consider foliations in the following sense:
\begin{definition}
	By a foliation of a surface with boundary, we mean a decomposition of $\Sigma$ into subsets which can be extended, in each local chart $U \xrightarrow{\cong} V \subset \mathbb R \times \mathbb R_{\ge 0}$, to a smooth 1-dimensional foliation of an open subset of $\mathbb R^2$ containing $V$.
	\medskip
	
	If marked points $V\subset \partial \Sigma$ are chosen, a foliation is moreover required to be tangent to the boundary at each marked point. A homotopy of foliations is required to preserve this tangency.
\end{definition}
Note that we don't put any requirements on the foliation away from $V$, i.e. it can be tangent to the boundary also away from $V$. For an example and classification of such foliations, 
see Appendix \ref{app:foliation}. 
\medskip

We can use a foliation to measure the number of turns of a path connecting points of $V$. 
\begin{definition}\label{def:foliation}
 Let $\gamma$ be an immersed path connecting two points of $V$, transverse
 to the boundary at its endpoints. Then, we define $\mathrm{rot}_{\gamma} \in \frac 12 \mathbb Z$  as the number of positive turns (with respect to the orientation of the surface) the foliation takes along the path.
 
 Concretely, for generic $\gamma$, there is a contribution of $\pm \frac 12$ for each
 time $\gamma$ becomes tangent to the foliation, with $-\frac 12$ if the turn is 
 compatible with the orientation of the surface, as on Figure \ref{fig:foliation}. Alternatively, we can  pick a metric on the surface such that $\gamma$ is perpendicular to the foliation  at its endpoints. Then the angle between $\dot{\gamma}$ and the foliation gives a closed  loop in $\mathbb{RP}^{1}$; the rotation numbers is one half of the homotopy class of this map, where we take as the generator the positive homotopically-trivial loop.
\end{definition}	
\begin{figure}[h]
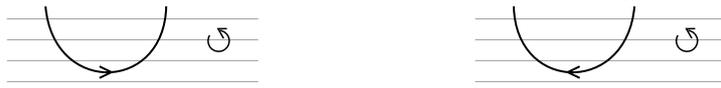

	\centering
	\begin{subfigure}[b]{0.35\textwidth}
		\begin{center}
			\includesvg{images/rotation1.svg}
		\end{center}
		\caption{Contribution of $-\frac12$ to the rotation number}
		\label{fig:foliation1}
	\end{subfigure}
	\qquad  
	\begin{subfigure}[b]{0.35\textwidth}
		\begin{center}
			\includesvg{images/rotation2.svg}
		\end{center}
		\caption{Contribution of $\frac12$ to the rotation number}
		\label{fig:foliation2}
	\end{subfigure}
	\caption{Rotation number. The surface is oriented counter-clockwise, as depicted.}\label{fig:foliation}
\end{figure}

The number of rotations $\mathrm{rot}_{\gamma}$ is an invariant of regular homotopy of $\gamma$. Moreover, it depends only on the homotopy class of the compatible foliation.
It satisfies
\[ \mathrm{rot}_{\gamma^{-1}} = - \mathrm{rot}_{\gamma}, \]
where $\gamma^{-1}$ is the path $\gamma$ with reversed orientation. 

Given two regular homotopy classes of paths meeting at a vertex, we can use the orientation of the surface to define their composition, see Figure \ref{fig:comprot}. With this convention, the rotation number of the composition $\gamma_2\gamma_1$ satisfies  \begin{equation}\label{eq:comprot}\mathrm{rot}_{\gamma_{2}\gamma_{1}} = 	\mathrm{rot}_{\gamma_{1}} + 	\mathrm{rot}_{\gamma_{2}}  + \frac 12.\end{equation} 
\begin{figure}[h]
    \centering
    \includesvg{images/comprotf.svg}
    \caption{The composition $\gamma_2\gamma_1$ is defined by smoothing the usual composition $\gamma_2\gamma_1$, where we position the ends of $\gamma_1$ and $\gamma_2$ as shown above (i.e. the end of $\gamma_1$ is before the start of $\gamma_2$ using the linear order defined in Figure \ref{fig:linorder}).}
    \label{fig:comprot}
\end{figure}


\subsection{The Poisson bivector of Fock and Rosly} \label{ssec:FR}
	Let us now recall the construction of Fock and Rosly \cite{FockRosly},
	of a Poisson structure on the moduli space of flat connections on a surface.
	\medskip

	As before, let $\Sigma$ be an oriented, compact surface with boundary, and
	$V$ a non-empty, finite set of marked points belonging to $\partial \Sigma$. Let $G$ be a connected 
	Lie group, with a Lie algebra $\g$. We will study $\ModSpace{\Sigma, V}{G}$, the space of principal 
	$G$-bundles with a flat connection and a chosen trivialization at $V$, modulo isomorphisms.
	Concretely, we will describe the moduli space as the space of all 
	groupoid homomorphisms from $\Pi_{1}(\Sigma, V)$  to $G$
	\[ \ModSpace{\Sigma, V}{G} = \mathrm{Hom}(\Pi_{1}(\Sigma, V), G) \,.\]
	Given a flat connection on $\Sigma$, we get an element of $\mathrm{Hom}(\Pi_{1}(\Sigma, V), G)$ assigning 
	 holonomy $\mathrm{hol}(\gamma)\in G$ to any path $\gamma$ in $\Sigma$ 
	 between any two points of $V$. This gives an isomorphism between the 
	 moduli space of flat bundles and $\mathrm{Hom}(\Pi_{1}(\Sigma, V), G)$. 
	\medskip{}

	If we choose a skeleton $\Gamma$ of $\Sigma$ with $N$ edges $\gamma_{1}, \dots, \gamma_{N}$, 
	we get an isomorphism
	\[ \Psi_{\Gamma} \colon G^{\times N} \xrightarrow{\sim}\ModSpace{\Sigma, V}{G} \]
	given by specifying the holonomies $(\mathrm{hol}(\gamma_{1}), \dots, 
	\mathrm{hol}(\gamma_{N})) \in G^{{\times N}}$. Note that in our convention,
	 composition of two paths $\gamma_{1}\gamma_{2}$ is a path first traversing
	  $\gamma_{2}$ and then $\gamma_{1}$ so that $\mathrm{hol}(\gamma_{1}\gamma_{2})
	   = \mathrm{hol}(\gamma_{1})\mathrm{hol}(\gamma_{2})$.
	\medskip{}

	On the moduli space of flat connections, there are $|V|$ pairwise-commuting 
	$G$ (and $\g$) actions $\rho_{p}$, coming
	from gauge transformations at points $p$ of $V$. They correspond to left/right
	multiplication of the holonomies of paths incident at that vertex. For example, for $V=\{p\}$,
	the $G$-action is $$(\mathrm{hol}(\gamma_{1}), \dots, \mathrm{hol}(\gamma_{N})) \mapsto{}
	(g \,\mathrm{hol}(\gamma_{1})g^{-1}, \dots, g\,\mathrm{hol}(\gamma_{N})g^{-1}).$$
	
	To define the bivector field of Fock and Rosly on the moduli space, we 
	need to choose a skeleton $\Gamma$ on $\Sigma$. Let us denote
	by $\he(p)$ the set of half-edges
	incident at $p\in V$, which is linearly ordered using the orientation of 
	$\Sigma$, see Figure \ref{fig:linorder}.
	
	If $\Gamma$ has edges $\gamma_{1}, \dots, \gamma_{N} $, we can define an action of $G^{\he(p)}$ on the moduli space $\ModSpace{\Sigma, V}{G}\cong G^{\times N}$ as follows. If $a$ is a half-edge of an edge $\gamma_{i}$, 
	and $g\in G$, we define  $(g)_{a}\colon G^{\times N} \to G^{\times N}$ by 
	\begin{equation}
	\begin{aligned} \label{eq:LAaction}
		(g)_{a} &= 1_{G}\times \dots \times L_{g} \times \dots \times 1_{G} \quad \text{ if $a$ arrives at $p$,} \\
		(g)_{a} &= 1_{G}\times \dots \times R_{g^{-1}} \times \dots \times 1_{G} \quad \text{ if $a$ leaves $p$,}
	\end{aligned}
	\end{equation}
	where $L_{g}$ or $R_{g^{-1}}$ act on the $i$th factor of $G^{\times N}$, 
	i.e. the factor corresponding to $\mathrm{hol}(\gamma_{i})$.
	For $x\in \g$, we denote the induced Lie algebra action $(x)_{a}$, i.e.
	$(x)_{a} = -x^{R}$ or $x^{L}$ for incoming/outgoing half-edge $a$. Note that
	the action $\rho_{p}$ is the product/sum of these actions 
	over all half-edges incident at $p$, for example $\rho_{p}(x) = \sum_{a\in \he(p)} (x)_{a}$.

\begin{figure}
\centering
\begin{minipage}{0.4\textwidth}
		\centering
		\includesvg{images/actions.svg}
		\caption{Lie group and Lie algebra action associated to
		half-edges.}
		\label{fig:actions}
\end{minipage}%
\hspace*{0.1\textwidth}\begin{minipage}{0.4\textwidth}
		\centering
		\includesvg{images/linorder.svg}
		\caption{Orientation of the surface gives a cyclic order
		on half-edges at $p$. Using the boundary, we can pick the first
		half-edge.}
		\label{fig:linorder}
\end{minipage}
	\end{figure}
	
	\medskip
	
	\emph{Let us now also assume} that on $\g$, there's a nondegenerate, invariant 
	symmetric pairing, with inverse $s\in (\mathrm{Sym}^2{\g})^{\g}$. If $e_{i}$
	is a basis of $\g$, let us write\footnote{We use the Einstein summation convention throughout the paper.} $s = s^{ij}e_{i}\otimes e_{j}$.
	 Recall that the  Cartan element $\phi \in \bigwedge^{3}\g$ is defined as 
		$$\phi(\alpha, \beta, \gamma) = \frac 1 {24} \alpha([s^{\#}(\beta), s^{\#}(\gamma)]) \quad \text{for } \alpha, \beta, \gamma \in \g^{*},$$ where $s^{\#} \colon \g^{*} \to \g$ comes from the non-degenerate pairing on $\g$.
	 
	 Let us define a bivector
	on $G^{\times N}$
	\begin{equation}\label{eq:FRbivector}
		\pi_{\mathrm{FR}}^{\Gamma} := \sum_{p\in V}\sum_{\substack{a, b \in \he(p) \\ a<b }}  \frac 12 s^{ij} (e_{i})_{a} \wedge{}
		(e_{j})_{b},
	\end{equation}
	where we use the linear order of half-edges from Figure \ref{fig:linorder}. 
	Using the isomorphism $\Psi_{\Gamma} \colon G^{\times N} \xrightarrow{\sim}\ModSpace{\Sigma, V}{G}$ we get a bivector
	field $(\Psi_{\Gamma})_{*} \pi^{\Gamma}_{\mathrm{FR}}$ on the moduli space $\ModSpace{\Sigma, V}{G}$. 
	The following theorem then follows from the work of Fock and Rosly \cite{FockRosly}, see also \cite{AKSM, MassuyeauTuraev2012, Nie2013,LBSQuilted}.

	\begin{remark}\label{rmk:noninvertible_even}
		To define $\pi^\Gamma$, we don't have to start with an invariant pairing; a non-necessarily-invertible element $s\in (\mathrm{Sym}^{2}(\g))^{\g}$ is sufficient. Theorem \ref{thm:FR} below holds in this case as well (as noticed in e.g. \cite{LBSQuilted}). See also Remark \ref{rmk:noninvertible_odd}.
	\end{remark}
	
	\begin{theorem}\label{thm:FR}
		The bivector $\pi_{\mathrm{FR}}:=(\Psi_{\Gamma})_{*} \pi^{\Gamma}_{\mathrm{FR}}$ on $\ModSpace{\Sigma, V}{G}$ does not depend on the choice of the skeleton $\Gamma$ and is invariant under the $G$-action $\rho_{p}$ on $\ModSpace{\Sigma, V}{G}$ for each $p\in V$.
	
		Moreover, the Schouten bracket of $\pi_\mathrm{FR}$ satisfies
		\[ [\pi_\mathrm{FR}, \pi_\mathrm{FR}]/2 = \sum_{p\in V} \rho_{p}(\phi), \]
		where $\phi$ acts as a trivector field on the moduli space, using the $\g$-action $\rho_{p}$.
		In other words, $\pi_\mathrm{FR}$ is a $\g^{V}$-quasi-Poisson bivector on the moduli space  \cite{AKS,LBSQuilted}.
	\end{theorem}
	\begin{remark}
		Fock and Rosly also fix a classical $r$-matrix, i.e. an element
		$\Lambda \in \bigwedge^{2} \g$ satisfying $[\Lambda, \Lambda]/2 = -\phi$. 
		Then, the bivector $\pi_\mathrm{FR} + \sum_{p\in V}\rho_{p}(\Lambda)$ is Poisson, see Eq.~(4.8) and Prop~4.1 of \cite{FockRosly}. 
        The independence of this Poisson structure on the skeleton is contained in \cite[Prop.~4.2]{FockRosly}, introducing special moves connecting skeletons (such as slide moves in the present work). The proof is omitted as a straightforward calculation in both cases in \cite{FockRosly}.
		
		A different way to obtain a Poisson structure is to look at the character variety $\ModSpace{\Sigma, V}{G}/G^{V}$, i.e. the moduli space of flat $G$-connections on $\Sigma$, with no marked points (see e.g. \cite[Prop.~4.3]{FockRosly}).
 	\end{remark}

 	We will now reprove this theorem, since a similar proof will be used in the Section \ref{sec:qBVFR}.

	\subsection{Proof of Theorem \ref{thm:FR}}\label{sec:qBVFRproof}
	Denote $$\g^{\he(p)} = \bigoplus_{a \in \he{(p)}}\g$$ the Lie algebra acting at each
	marked point. Let $\iota_{a} \colon \g \to \g^{\he(p)}$ be the inclusion associated with 
	half-edge $a$.
	Recall that the Lie algebra actions $x \mapsto (x)_{a}$ extend to a morphism
	from $\bigwedge \g^{\he{(p)}}$ to multivector fields on $G^{\times N} $,
	compatible with the wedge product and the Schouten 
	brackets\footnote{i.e. a morphism of Gerstenhaber algebras}.  For example, the element
	$\iota_{a}(x)\in \g^{\he(p)} \subset \bigwedge \g^{\he(p)}$ is sent to 
	the  vector field $(x)_{a} \in \mathfrak X (G^{\times N})$

	Moreover, if $a, b$ are two half-edges, we can define  
	$\tilde{s}_{ab} = s^{ij} \iota_{a}(e_{i})\wedge \iota_{b}(e_{j})$,
	so that the bivector $\pi^{\Gamma}_{\mathrm{FR}}$ is given by the image of 
	$$\sum_{p\in V} \sum_{\substack{a, b \in \he(p) \\ a<b }} \frac12 \tilde{s}_{ab}
	 \in \bigwedge \left( \bigoplus_{p\in V}\g^{\he(p)} \right).$$

	Similarly, we define $\tilde{\phi}_{abc} = \frac 1 {24} f^{ijk} \iota_a(e_i) \wedge \iota_b(e_j) \wedge \iota_c(e_k )$,
	where $\phi = \frac 1 {24} f^{ijk} e_{i} \wedge e_{j} \wedge e_{k} \in \bigwedge^{3}\g $ uses the antisymmetric
	Cartan tensor $f^{ijk} = f_{xy}^{i}s^{xj}s^{yk}$ defined using the structure constants $[e_{x}, e_{y}] = f^i_{xy} e_i$ of $\g$. These elements satisfy a version of the Drinfeld-Kohno relations under the Schouten bracket, see Proposition \ref{prop:sproperties} in the Appendix \ref{app:evenLie}. Now we can prove the theorem, mostly on the level of $\bigwedge \g^{\he{(p)}}$.
\begin{proof}[Proof of Theorem \ref{thm:FR}]
	From the invariance of the inner product it follows that  $[(x)_{a} + 
	(x)_{b}, \tilde{s}_{ab}] = 0$ for each $a, b$ at $p$, and thus $\tilde{s}_{ab}$
	is invariant.
	\medskip{}

	To prove $[\pi^{\Gamma}_{\mathrm{FR}}, \pi^{\Gamma}_{\mathrm{FR}}]/2 = \sum_{p}\rho(\phi)_{p}$,
	we first look at each vertex $p$ separately.
	The action of $\phi$ at a vertex $p$ is given by the action of
	\[ \sum_{a, b, c\in \he(p)}\tilde\phi_{abc} = \sum_{a \in \mathrm{he}(p)} \tilde\phi_{aaa} 
	+ 3\!\!\!\!\!\!\sum_{a, b \in \mathrm{he}(p), a<b}\!\!\!\! \left( \tilde\phi_{aab} + \tilde\phi_{abb}\right) 
	+ 6\!\!\!\!\!\!\!\!\!\!\sum_{a, b, c \in \mathrm{he}(p), a<b<c} \!\!\!\!\!\!\!\!\tilde\phi_{abc}\,.\]
	Using Proposition \ref{prop:sproperties}, we get that the term
	$[\frac12\tilde{s}_{ab}, \frac12\tilde{s}_{ab}]/2$ will cancel the term 
	$3(\tilde\phi_{aab} + \tilde\phi_{abb})$ and that $6\tilde\phi_{abc}$
	cancels with 
	\[ [\tfrac12\tilde{s}_{ab}, \tfrac12\tilde{s}_{bc}] + [\tfrac12\tilde{s}_{ab}, \tfrac12\tilde{s}_{ac}] + [ \tfrac12\tilde{s}_{bc}, \tfrac12\tilde{s}_{ac} ] = -6\tilde{\phi}_{abc} + 6\tilde{\phi}_{bac} + 6\tilde{\phi}_{bca}\,. \]

	Finally, for a path $\gamma$ with half-edges $a$ and $b$, one has
	$x_{a} = -(\Ad_{\mathrm{hol}_{\gamma}}x)_{b}$ and using the invariance of $\phi$, the 
	term $\tilde{\phi}_{aaa}$ thus cancels with $\tilde{\phi}_{bbb}$.
	\medskip{}

\begin{figure}
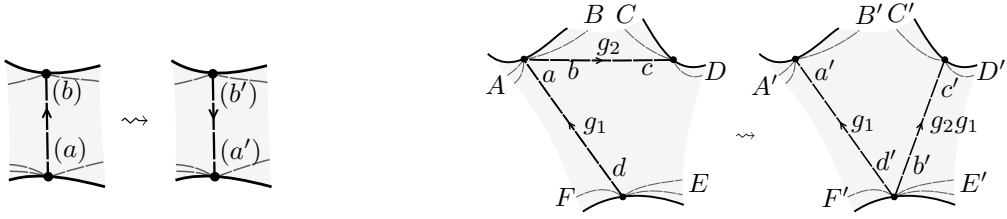

	\centering
	\begin{subfigure}[b]{0.3\textwidth}
		\begin{center}
			{\footnotesize \includesvg{images/inverselabelled.svg}}
			\vspace*{8mm}
		\end{center}
		\caption{Half-edge actions $a$ and $b$ get mapped to $a'$ and $b'$.}
		\label{fig:FRlabelled_slides_inverse}
	\end{subfigure}
	\qquad  
	\begin{subfigure}[b]{0.62\textwidth}
		\begin{center}
			{\footnotesize \includesvg[scale=0.7]{images/slidelabelled.svg}}
		\end{center}
		\caption{The capital letters can correspond to more half-edges, the action $x_{A}$, for $x\in \g$,
			is given as a sum of individual half-edge actions.}
		\label{fig:FRlabelled_slides_slide}
	\end{subfigure}
	\caption{Choice of names for half-edges. The orientation of the surface is counter-clockwise.}\label{fig:FRlabelled_slides}
\end{figure}

	To show that $\pi_{\mathrm{FR}}$ is independent of $\Gamma$, it is enough
	to show that the diffeomorphism $\Phi = \Psi_{\Gamma'}^{-1}\Psi_{\Gamma}$ 
	sends $\pi^{\Gamma}_{\mathrm{FR}}$ to $\pi^{\Gamma'}_{\mathrm{FR}}$. 
	Moreover, we just need to check this on edge reversions 
	and slides of skeletons, c.f. Proposition \ref{prop:2Dtopology}.
	The isomorphism $\Phi\colon G^{\mathrm{edges}(\Gamma)} \to G^{\mathrm{edges}(\Gamma')}$,
	corresponding to the change of edge orientation, satisfies
	$$\Phi_*((x)_{a}) = (x)_{a'} \qquad \text{and} \qquad \Phi_*((x)_{b}) = (x)_{b'}\,,$$
	with labels for half-edge as in Figure \ref{fig:FRlabelled_slides_inverse}. This is because
	$\mathrm{Inv} \colon g \mapsto g^{-1}$ satisfies $\mathrm{Inv}_{*}(x^{\mathrm L}) = - x^{\mathrm R}$. 
	Thus, $\Phi_*$ relates the two bivector fields associated to $\Gamma$ and $\Gamma'$, because
	$\Phi_{*}(s_{ab}) = s_{a'b'}$ and the two bivector fields are equal term-by-term.
	\medskip
	
	For the case of the slide move of Figure \ref{fig:slide}, we denote 
	the half-edge as on Figure \ref{fig:FRlabelled_slides_slide}.
	The diffeomorphism $\Phi: G^{\mathrm{edges}(\Gamma)} \to G^{\mathrm{edges}(\Gamma')}$ 
	is in this case (on the relevant holonomies) given as $(g_{1}, g_{2})\mapsto (g_{1}, g_{2}g_{1})$.
	Thus, on vector fields it gives
	\begin{align*}
	\Phi_*((x)_{a}) &= (x)_{a'} + (\Ad_{g_2}(x))_{c'}\\
	\Phi_*((x)_{b}) &= - (\Ad_{g_2}x)_{c'} = (\Ad_{g_1^{-1}} x)_{b'} \\
	\Phi_*((x)_{c}) &= (x)_{c'} \\
	\Phi_*((x)_{d}) &= (x)_{d'} + (x)_{b'} \\
	\Phi_*((x)_{X}) &= (x)_{X'} \quad \text{ for any uppercase $X$} 
	\end{align*}
	The relevant part (dropping terms containing only $s_{XY}$ for $X, Y$ uppercase) of the quasi-Poisson bivector field for the original graph is
	\[ s_{A(a+b)} + s_{(a+b)B} + s_{ab} + s_{Cc} + s_{cD} + s_{Ed} + s_{dF}  \]
	and under $\Phi_*$, it gets mapped to
	\[  s_{A'a'} + s_{a'B'} + s_{(a'+\Ad_{g_2}c')(-\Ad_{g_2} c')} + s_{C'c'} + s_{c'D'}+ s_{E'(b'+d')} + s_{(b'+d')F'}. \]
	Using $(x)_{a'} = - (\Ad^{-1}_{g_1}(x))_{d'}$ and properties of $s$, we get 
	\begin{equation}\label{eq:stransf} s_{(a'+\Ad_{g_2}c')(-\Ad_{g_2} c')} 
	=  s_{a'(\Ad^{-1}_{g_1} b')} = s_{(\Ad_{g_1}a')b'} = - s_{d'b'} = s_{b'd'}\,.
	\end{equation}
	which gives the quasi-Poisson bivector field $\pi^{\Gamma'}_{\mathrm{FR}}$.
	Note the term $s_{\Ad_{g_2}c', -\Ad_{g_2}c'}$ which is zero; it will be
	non-trivial in the definition of the quasi-BV operator. 
	\medskip

	We also need to consider a 
	case where some of the marked points on Figure \ref{fig:moves} coincide. 
	In other words, it might happen that one of the two edges involved in a slide move is a loop. Instead of repeating the above calculation, we recall the so-called fusion procedure.
	
	\begin{figure}[h]
		\centering
		\captionsetup{width=0.7\linewidth}
		{\includesvg{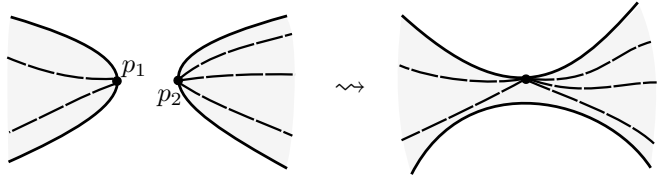}}
		\caption{Fusion at vertices $p_1$ and $p_2$. Orientation of $\Sigma$
		determines the position of the new point.}
		\label{fig:FRfusion}
	\end{figure}
	Fusion of a surface at two points $p_1$ and $p_2$ is given by a 
	corner-connected sum of the (possibly disconnected) surface, with the two marked points
	replaced by one point as on Figure \ref{fig:FRfusion}. If we start with a surface
	with a skeleton $\Gamma$, the fused surface also has a natural skeleton $\Gamma_{\mathrm{fused}}$, and conversely
	any skeleton with a loop can be obtained by fusion which creates that loop.
	This becomes evident when the surface is seen as a fattening of its skeleton,
	as on Figure \ref{fig:FusionFRLoop}.
	
	\begin{figure}[h]
	\centering
	\includesvg[scale=1]{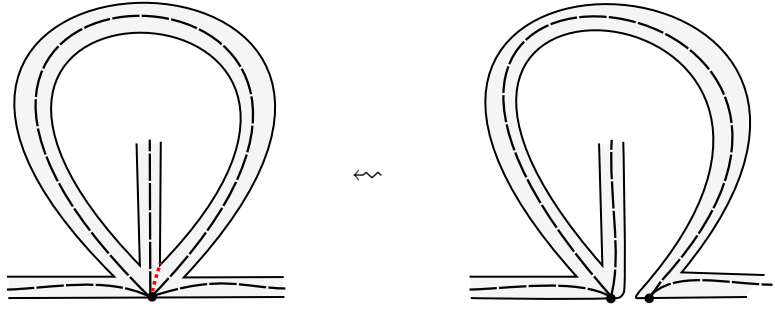}
	\caption{A loop in a skeleton can be obtained via fusion.}
	\label{fig:FusionFRLoop}
	\end{figure}
	
	 As manifolds,
	the two moduli spaces are the same, however, the bivector 
	on the fused surface has an additional term, given by
 \begin{equation*} \pi^{\Gamma_{\mathrm{fused}}}_{\mathrm{FR}} 
	 = \pi^{\Gamma}_{\mathrm{FR}} + \frac 12 s^{ij} \rho_{p_{1}}(e_{i})\wedge \rho_{p_{2}}(e_{j})\,. \end{equation*}
  Here $\rho_{p}$ denotes the $\g$-action at the vertex $p$, i.e. $\rho_{p}(x) = \sum_{a\in\he(p)}(x)_{a}$.
	If $\Gamma$ and $\Gamma'$ are two skeletons related by slide not involving loops, we know the corresponding bivector fields
	$(\Psi_{\Gamma})_{*} \pi^{\Gamma}_\mathrm{FR}$ and $(\Psi_{\Gamma'})_{*} \pi^{\Gamma'}_\mathrm{FR}$ 
	are equal for the unfused surface. 	However, the fusion term 
	\[ \frac 12 s^{ij} \rho_{p_{1}}(e_{i})\wedge \rho_{p_{2}}(e_{j})\]
	is independent of $\Gamma$. Thus, the bivector fields  $(\Psi_{\Gamma_{\mathrm{fused}}})_{*} \pi^{\Gamma_{\mathrm{fused}}}_\mathrm{FR}$ and $\pi^{\Gamma'_{\mathrm{fused}}}_\mathrm{FR}$ are also equal, where now $\Gamma_{\mathrm{fused}}$ and $\Gamma'_\mathrm{fused}$ are related by a slide involving a loop.
 \end{proof}

\subsection{Goldman Lie algebra} \label{ssec:Goldman}
If we choose an $\Ad$-invariant function $f$ on $G$ , there is a distinguished set of functions $f_{|\gamma|}\colon [\rho] \mapsto f(\rho(\gamma))$ on the moduli space $\ModSpace{\Sigma}{G}$, where $|\gamma|$ is a free homotopy class of loops\footnote{A free homotopy class a loop is a map $S^1 \to \Sigma$, with two such maps identified if they can be extended to a map from the cylinder $S^1\times [0,1]$. Equivalently, $\pi_1^\text{free}(\Sigma)$ is the set of conjugacy classes of $\pi_1(\Sigma)$, and furthermore $\Goldman{\Sigma} \cong \mathbb R \pi_1(\Sigma)/[\mathbb R \pi_1(\Sigma), \mathbb R \pi_1(\Sigma)]$.} on $\Sigma$ and $\gamma$ its arbitrarily chosen representative in $\pi_1(\Sigma)$. In many cases, the Poisson bracket of these functions was described by Goldman \cite{Goldman}; let us recall the case of $GL_n(\mathbb R)$ with $f=\tr$.

Denote by $\Goldman{\Sigma} = \mathbb R \pi_1^\text{free}(\Sigma)$ the vector space generated by free homotopy classes of loops  on $\Sigma$, and by $[-, -]_\mathrm{G}$ the Goldman Lie bracket on $\Goldman{\Sigma}$, given by resolving intersections of loops (see \cite[Section~3.13]{Goldman} or Section \ref{ssec:queerGT} below for details).
\begin{theorem}[{\cite{Goldman}}]\label{thm:Goldman}
	Let $G = GL_n(\mathbb R)$ and let $$\egm \colon \Goldman{\Sigma} \to C^{\infty}(\ModSpace{\Sigma}{GL(n)}), \; \egm_{|\gamma|} ([\rho]) := \tr_{|\gamma|}([\rho])=\tr(\rho(\gamma)).$$
 Then $\egm$ is well defined and  $$\egm_{[|\gamma_1|, |\gamma_2|]_\mathrm{G}} = \{ \egm_{|\gamma_1|}, \egm_{|\gamma_2|} \}_{\mathrm{AB}},$$ i.e. $\egm$ is a map of Lie algebras. Here, $\{ \egm_{|\gamma_1|}, \egm_{|\gamma_2|} \}_{\mathrm{AB}}$ denotes the Poisson bracket induced by $\pi_\mathrm{FR}$.
\end{theorem}
We will describe a slight generalization of Goldman's theorem. Our motivation is to also capture the determinant, and we will extend the Goldman Lie algebra by the first homology $H_1(\Sigma)$ to achieve this. 
\begin{definition}\label{def:Goldmandet}
	On $\Goldman{\Sigma} \oplus H_1(\Sigma, \mathbb R)$, extend the Goldman Lie bracket by
	\begin{align*}
	[|\gamma|, a]_\mathrm{G} 		&:= \langle [\gamma], a\rangle |\gamma|, \\
	[a, b]_\mathrm{G}				&:= \langle a, b \rangle \bigcirc,
	\end{align*}
	where $|\gamma| \in \pi_1^\text{free}(\Sigma)$, $a, b\in H_1(\Sigma, \mathbb R)$, $\langle a,b\rangle$ is the intersection pairing on $H_1(\Sigma, \mathbb R)$, $[\gamma]\in H_1(\Sigma, \mathbb R)$ is the homology class given by the free homotopy class $|\gamma|$ and $\bigcirc\in \Goldman{\Sigma}$ is the homotopy class of the constant loop.
\end{definition}


It is straightforward to check that the above bracket on $\Goldman{\Sigma} \oplus H_1(\Sigma, \mathbb R)$ satisfies the Jacobi identity.
Let us now also extend the map $ \egm \colon \Goldman{\Sigma} \to C^{\infty}(\ModSpace{\Sigma}{G})$. 
\begin{definition}
	Let $G = GL_n(\mathbb R)^+$, the connected component of the identity.
	Define $\egmex \colon \Goldman{\Sigma} \oplus H_1(\Sigma, \mathbb R) \to C^{\infty}(\ModSpace{\Sigma}{G})$ by
	\begin{align*}
		\egmex_\gamma([\rho]) &= \tr (\rho(\gamma)), \\
		\egmex_a ([\rho]) &=  \log \det (\rho(\gamma_a)),
	\end{align*}
	where $|\gamma|\in \pi_1^\text{free}(\Sigma)$ and $\gamma_a\in\pi_{1}(\Sigma)$ is a representative of $a\in H_1(\Sigma)\cong \pi_{1}(\Sigma)^\mathrm{ab}$. 
\end{definition}

\begin{proposition}
	The map $\egmex$ is well defined.
\end{proposition}
\begin{proof}
	We need to check that two different representatives $\gamma_a, \gamma_a'$ of $a$ give the same function. This follows from the fact that $\gamma_a = \gamma_a' C$, where $C$ is a product of commutators, and the determinant of a commutator is equal to $1$.

    We can extend $\egmex$ from $H_1(\Sigma, \mathbb Z)$ to  $H_1(\Sigma, \mathbb R)$, since $\egmex$ is a map of abelian groups. This follows from the fact that $\gamma_a\gamma_b$ is a representative for $a+b$, and $\log\det \hol_{\gamma_a\gamma_b} = \log \det \hol_{\gamma_a} + \log \det \hol_{\gamma_b}$.
\end{proof}
\begin{theorem}
	The map $\egmex \colon \Goldman{\Sigma} \oplus H_1(\Sigma, \mathbb R) \to C^{\infty}(\ModSpace{\Sigma}{G})$ is a morphism of Lie algebras.
\end{theorem}
\begin{proof}
	Instead of Equation \eqref{eq:FRbivector}, it is more convenient to use a description of the Poisson structure on $\ModSpace{\Sigma}{G}$ as in \cite[p.~265, \textit{Product formula}]{Goldman}. Namely, if $f, f'$ are two $\Ad$-invariant functions on $G$, then\footnote{See also \cite[Prop.~4.3]{FockRosly}, \cite[Comment~18]{Lawton2008}, \cite[Theorem~2.5]{Nie2013} and \cite[Prop.~4]{LBSQuilted} for various versions of this claim; we prove a similar statement in the odd case in Theorem \ref{thm:geometricBV}. It is not difficult to check that our conventions for $\pi_{\mathrm{FR}}$ do indeed match that for the product formula above, by considering e.g. the 1-punctured torus.}
	\[ \{f_{|\gamma_1|}, f'_{|\gamma_2|} \}_\mathrm{AB} = \!\sum_{p \in \gamma_1\cap\gamma_2}\!\!\!\pm s^{ij} \partial_{t_1}|_0 f((A_1)_p e^{t_1 e_i})\partial_{t_2}|_0 f'((A_2)_p e^{t_2 e_j}),\]
	where the sum is over all (transverse double) intersections of $\gamma_1$ and $\gamma_2$, and $(A_1)_p$, $(A_2)_p$ are holonomies along $\gamma_1$, $\gamma_2$ starting at $p$. The sign is given by the orientation of $((\dot{\gamma_1})_p, (\dot{\gamma_2})_p)$ relative to the orientation of $\Sigma$.
	
	Let us only do the calculation for the simpler case of $GL_n(\mathbb R)$. The basis is given by elementary matrices $E_{(\alpha\beta)}, \alpha, \beta= 1\dots n$, and $s = \sum_{\alpha, \beta} E_{(\alpha\beta)} \otimes E_{(\beta\alpha)}$. If $f= \tr$, we have
	\[  \partial_{t}|_0 \tr(A_pe^{t E_{(\alpha\beta)}}) = \tr(A_p E_{(\alpha\beta)}) = (A_p)_{\beta\alpha}, \]
	while for $f = \log\det$, we have
	\begin{equation}\label{eq:derlogdet}  \partial_{t}|_0 \log\det(A_p e^{t E_{(\alpha\beta)}}) =  \partial_{t}|_0 \log\det(e^{t E_{(\alpha\beta)}}) = \tr(E_{(\alpha\beta)}) = \delta_{\alpha, \beta}. \end{equation}	Thus, using the product formula, we get
	\begin{align*} \{\tr_{|\gamma|}, {\log\det}_{|\gamma_a|} \}_\mathrm{AB} =& \!\!\sum_{p \in \gamma\cap\gamma_a} \!\!\!\pm (A_p)_{\beta\alpha} \delta_{\beta,\alpha} \\=& \!\!\sum_{p \in \gamma\cap\gamma_a} \!\!\!\pm \tr_{|\gamma|} = \langle [\gamma], a\rangle \tr_{|\gamma|} \end{align*}
	and	
	\[ \{{\log \det}_{|\gamma_a|}, {\log \det}_{|\gamma_b|} \}_\mathrm{AB} = \!\sum_{p \in \gamma_a\cap\gamma_b} \!\!\!\pm  \delta_{\alpha, \beta} \delta_{\beta,\alpha} = \langle a, b \rangle n =  \langle a, b \rangle \tr_{\bigcirc}.\]
\end{proof}
\begin{remark}
	For simplicity, we described the Goldman theorem as stated in \cite{Goldman}, without marked points on $\Sigma$. See \cite{MassuyeauTuraev2012} and \cite{Nie2013} for a version with marked points.
\end{remark}
\begin{remark}
	One can consider the group $GL_n(\mathbb C)$ as well. Then, seen as a real group, the above theorem holds with $f=2 \operatorname{Re} \tr$ and with $\operatorname{Re}\log\det$ instead of $\log\det$. 
	
	Alternatively, one can define a holomorphic bivector field as in Equation \eqref{eq:FRbivector}, using left-invariant holomorphic vector fields. Taking $f= \tr$ and replacing $GL_n(\mathbb C)$ with its universal cover to define $\log\det$, the above theorem then holds as well.
\end{remark}

\section{BV operators on moduli spaces}
In this section, we prove an analogue of Theorem \ref{thm:FR} for Lie supergroups equipped with an odd pairing on their Lie algebras.

\subsection{Lie superalgebras with an odd pairing}To get a Batalin-Vilkovisky structure on the moduli space, we will use a 
Lie superalgebra $\g$ with an odd invariant pairing. 
\begin{definition}\label{def:oddinvariantpairing}
	If $\g$ is a Lie superalgebra, an odd invariant pairing is a graded-symmetric,
	non-degenerate odd map $\langle\,,\,\rangle \colon \g\otimes \g \to \fld$ satisfying
	\[  \langle[x, y], z\rangle + (-1)^{\hdeg x \hdeg y } \langle y, [x, z] \rangle=0 , \quad \forall x, y, z \in \g . \]
	A Lie superalgebra with such pairing is called an \emph{odd metric Lie salgebra}.
	This pairing defines a $\g$-equivariant isomorphism\footnote{$\Pi$ denotes the one-dimensional super vector space $\mathbb R^{0|1}$ concentrated in the odd degree.} $t^{\flat}\colon 	\Pi \ot \g \to \g^{*}$ by $t^{\flat}(\Pi \ot x)(y) = \langle x, y \rangle$, 
	whose inverse we will denote $t^{\#}$.
	
	The (odd version of the) Cartan element is defined\footnote{The sign $(-1)^{\beta}$ is because the odd
	map $t^{\#}$ passes through it.} as
	$$ \phi(\alpha, \beta, \gamma) = (-1)^{\hdeg \beta} \frac 1 {24}\alpha( [t^{\#}\beta, t^{\#}\gamma] ), $$
	where $\alpha, \beta, \gamma \in \g^{*}$.
	
	Finally, define an odd element $\nu\in \g$ by
	\[ \langle \nu, x \rangle = \mathrm{str}_{\g}\ad_{x}, \forall x \in \g. \]
\end{definition} 
	Recall that a Lie (super)algebra is called unimodular 
	if $\ad_{x}\colon \g \to \g$ is traceless for all $x\in \g$. By the 
	previous definition, a unimodularity of $\g$ is equivalent to vanishing of $\nu$. For more details on $t$, $\phi$ and their coordinate expressions, see Appendix \ref{app:oddLie}

\subsection{BV structure on the moduli of flat connections} \label{sec:qBVFR}
Let us now turn our attentions to moduli spaces of flat $G$-connections, with
$G$ a supergroup (see \cite{Varadarajan2004} and \cite{DeligneQuantumFields1999}).
\begin{definition}\label{def:oddmodulispace}
	Let $\Sigma$ be a compact, oriented surface with
	boundary and $V\subset \partial \Sigma$ a finite, non-empty set of marked points. 
	For a supergroup $G$, the moduli space of flat $G$-connections on $(\Sigma, V)$ is the supermanifold
\begin{equation}
	\ModSpace{\Sigma, V}{G} = \Hom_{SGrpd}(\Pi_{1}(\Sigma, V), G)\,.
\end{equation}
\end{definition}

There is a natural action of $G^{V}$ on this space, let us denote by $\rho_{p}(g)$ and
$\rho_{p}(x)$ the Lie group and Lie algebra actions of the $p$th factor.
In this section, we will define a second-order differential operator on $\ModSpace{\Sigma, V}{G}$
that will turn this space into a so-called $\g^{V}$-quasi BV manifold.
\begin{definition}\label{def:qBVmfld}
	A supermanifold $M$ with an action of an odd metric Lie algebra $\g$ is called
	$\g$-quasi-BV if it is equipped with an odd, second-order, $\g$-invariant differential operator $\Delta$ satisfying 
	\[ \Delta(1) = 0 \quad \text{and} \quad \Delta^{2} = \phi\,, \]
	where $\phi$ acts on $M$ as an element of $U{\g}$.
\end{definition}
If we choose a skeleton $\Gamma$ of $\Sigma$ with edges $\gamma_{1}, \dots, \gamma_{N}$,
we get an isomorphism
$$\Psi_{\Gamma}\colon G^{\times N}   \xrightarrow{\sim}  \ModSpace{\Sigma, V}{G} .$$ 
This way, the abstract space of all maps of groupoids gets a concrete
description, which we will mostly use.\footnote{The functor from supermanifolds
to sets, given by
\[ X \mapsto \mathrm{Hom}_{\mathrm{Groupoid}} ( \Pi_{1}(\Sigma, V), \mathrm{Hom}_{\mathrm{SuperMfld}}(X, G) ) \]
is representable by each of these moduli spaces constructed using $\Gamma$.
This gives an alternative definition of the moduli space $\ModSpace{\Sigma, V}{G}$.}

As in Section \ref{sec:qBVFR}, can define, for any half-edge $a$, an action
of the Lie algebra $\g$, denoted by $(x)_{a}$ for $x \in \g$, by the 
equation \eqref{eq:LAaction}, i.e. left-invariant or minus of the right invariant vector field on the
corresponding component of $G^{V}$.
\medskip{}

To define the quasi BV operator, we also need to fix a foliation of $\Sigma${}
as in Definition \ref{def:foliation}, i.e. we require that the foliation is tangent
to the boundary at the marked points. Let us also choose $\Gamma$ such that its edges $\gamma_{i}$ are transverse to the foliation at the marked points.

If $\g$ has odd invariant pairing as in Definition \ref{def:oddinvariantpairing}, denote by $t^{ij}$ the matrix inverse to 
the matrix of the pairing $t_{ij} = \langle e_{i}, e_{j}\rangle$ for a homogeneous basis $e_{i}$
of $\g$. Then we define
\begin{equation}\label{eq:qbvdef}
	\Delta^{\Gamma} = \sum_{p\in V}\sum_{\substack{a, b \in \he(p) \\ a<b }}  
	\frac 12 (-1)^{\hdeg{e_{i}}} t^{ij} (e_{i})_{a} (e_{j})_{b}
	+ \sum_{\gamma\in\Edges{\Gamma}} \frac 12 \mathrm{rot}_{\gamma} \cdot (\nu)_{a_\gamma}
\end{equation}
where $a_\gamma$ is the outgoing half-edge of $\gamma$.
\begin{remark}
	If $\g$ is unimodular, the second term of $\Delta^{\Gamma}$ is zero
	and we don't need the foliation of the surface. However, if $\g$ is not unimodular,
	just the first term of $\Delta^{\Gamma}$ would not be invariant under 
	slide moves, as we will see below.
\end{remark}

\begin{remark}\label{rmk:noninvertible_odd}
	As in the even case (see Remark \ref{rmk:noninvertible_even}), the definition of $\Delta^\Gamma$ and Theorem \ref{thm:qBVFR} below are valid also in the case
	one starts with $\g$ and an element $t = (-1)^{\hdeg{e_i}}t^{ij} e_i \wedge e_j \in(\bigwedge^{2}\g)^{\g}$, i.e.
	$t$ doesn't have to be the inverse of an odd, nondegenerate pairing on $\g$. 
\end{remark}

\begin{theorem}\label{thm:qBVFR}
	The operator  $\Delta :=(\Psi_{\Gamma})_{*}\Delta^{\Gamma}$ on $\ModSpace{\Sigma, V}{G}$ is independent of 
	the choice of $\Gamma$ compatible with the foliation and does not change
	under homotopy of the foliation. 

	Futhermore, $\Delta$ satisfies all the properties of Definition \ref{def:qBVmfld}
	and thus equips $\ModSpace{\Sigma, V}{G}$ with a structure of a $\g^{V}$-quasi BV
	manifold.
\end{theorem}
The proof will be similar to the proof in Section \ref{sec:qBVFRproof}. Now,
instead of multivector fields, we work with differential operators, and
thus the Gerstenhaber algebra $\bigwedge \g^{\he (p)}$ will be replaced by
the associative algebra $U\g^{\he(p)}$. 
\begin{proof}
Let us define $\tilde{t}_{ab}\in U\g^{\he (p)}$ by $(-1)^{\hdeg{e_{i}}} t^{ij} 
\iota_{a}(e_{i}) \iota_{b}(e_{j})$, where $\iota_{a}\colon \g \to \g^{\he(p)}$ 
is the inclusion into the $a$th copy.
Similarly, $\tilde\phi_{abc} = \phi^{xyz} \iota_{a}(e_{x})\iota_{b}(e_{y})\iota_{c}(e_{z})$ and $\tilde\nu_a = \iota_a(\nu)$ (see Proposition
\ref{prop:oddmetric} for the definition of $\phi^{xyz}$). The operator $\Delta^\Gamma$ can be then written as the action of
\[ 	\Delta^{\Gamma} = \sum_{p\in V}\sum_{\substack{a, b \in \he(p) \\ a<b }}  
	\frac 12 \tilde{t}_{ab}
	+  \sum_{\gamma\in\Edges{\Gamma}} \frac 12 \mathrm{rot}_{\gamma} \cdot \tilde\nu_{a_\gamma}.
	\]
These elements have properties similar to those of $\tilde{s}_{ab}$ and $\tilde{\phi}_{abc}$ 
in the even case, which are collected in Proposition \ref{prop:tproperties} in Appendix \ref{app:oddLie}. Using this result, we can just follow the proof of Theorem \ref{thm:FR} without many modifications.
For example, the invariance of $\Delta$ w.r.t. the action of $\g^{V}$ follows from the invariance
of $\tilde{t}$ and $\tilde\nu$. Since $\nu$ is central, it does not enter into the calculation of 
$\Delta^{2} = [\Delta, \Delta]/2$ and we can repeat the arguments of Section
\ref{sec:qBVFRproof} verbatim. 

The invariance of $\Delta$ under edge inversion is as before, using
$\mathrm{Inv}_{*}( (x)^{\mathrm L} ) = - (x)^{\mathrm R}$. The additional term $\nu_{a} = \nu^{L}$,
 is sent to $\mathrm{Inv}_{*}(\nu^{L}) = - \nu^{R} = - \nu^{L}$,
since $\nu$ is $G$-invariant. This minus sign is canceled by $\mathrm{rot}_{\gamma^{-1}} = -\mathrm{rot}_{\gamma}$.

For the slide,
the formulas expressing the action of $\Phi = \Psi^{-1}_{\Gamma'}\Psi_{\Gamma}$
on the vector fields are still correct.
However, the term $-\frac 12 t_{ \Ad_{g_{2}}(c'), \Ad_{g_{2}}(c')}$ from equation \eqref{eq:stransf} is not zero, but gives
$-\tilde{\nu}_{c'}/4 = \tilde{\nu}_{b'}/4$. 	
This counters the discrepancy between
	\[  \Phi_{*} \left(\frac{\mathrm{rot}_{\gamma_{1}}}{2}\nu_{d} + \frac{\mathrm{rot}_{\gamma_{2}}}{2}\nu_{b}\right) 
	=  \frac{\mathrm{rot}_{\gamma_{1}}}{2}\nu_{d'} + \frac{\mathrm{rot}_{\gamma_{1}}}{2}\nu_{b'} + \frac{\mathrm{rot}_{\gamma_{2}}}{2}\nu_{b'}\]
	and 
	\[  \frac{\mathrm{rot}_{\gamma_{1}}}{2}\nu_{d'} + \frac{\mathrm{rot}_{\gamma_{2}\gamma_{1}}}{2}\nu_{b'} 
	= \frac{\mathrm{rot}_{\gamma_{1}}}{2}\nu_{d'} + \frac{\mathrm{rot}_{\gamma_{1}}+ \mathrm{rot}_{\gamma_{2}} + \frac 12}{2}\nu_{b'}\,. \]  
where $\gamma_{i}$ is the path with holonomy $g_{i}$ in Figure \ref{fig:FRlabelled_slides_slide}. Here, we used that the composition in Figure \ref{fig:FRlabelled_slides_slide} agrees with the convention in Figure \ref{fig:comprot}, i.e. \eqref{eq:comprot} holds.

Finally, the fusion works as before. The foliation on the fused surface is extended as on Figure
\ref{fig:foliatedfusion}.
	\begin{figure}[h]
		\centering
		{\includesvg{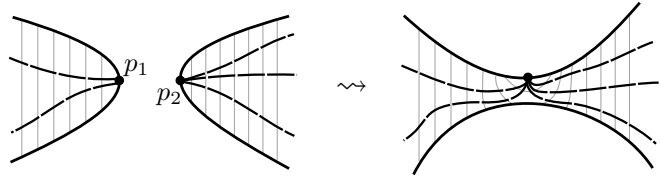}}
		\caption{Fusion of a foliated surface at vertices $p_1$ and $p_2$.}
		\label{fig:foliatedfusion}
	\end{figure}
The paths acquire no additional rotation with respect to this new foliation, so 
we again have that
\begin{equation} \label{eq:fusionBV} \Delta^{\Gamma_{\mathrm{fused}}}
	 = \Delta^{\Gamma} + \frac 12 (-1)^{\hdeg{e_{i}}} t^{ij} \rho_{p_{1}}(e_{i})\rho_{p_{2}}(e_{j})\,, \end{equation}
	where $\rho_{p}$ denotes the $\g$-action at the 
	vertex $p$, i.e. $\rho_{p}(x) = \sum_{a\in\he(p)}(x)_{a}$.

Then if we have a loop in $\Gamma$, we can always see the surface as a fusion
of a different surface, where the loop is split. Moreover, the surface retracts to a thickening of its 
skeleton, on which the foliation is as on  Figure \ref{fig:FusionBVLoop} (see Appendix \ref{app:foliation}). Thus, the foliation on
the fused surface can be obtained from the foliation of the unfused surface by
deformation, with rotation numbers unchanged.

\begin{figure}[h]
	\centering
	\includesvg[scale=1.2]{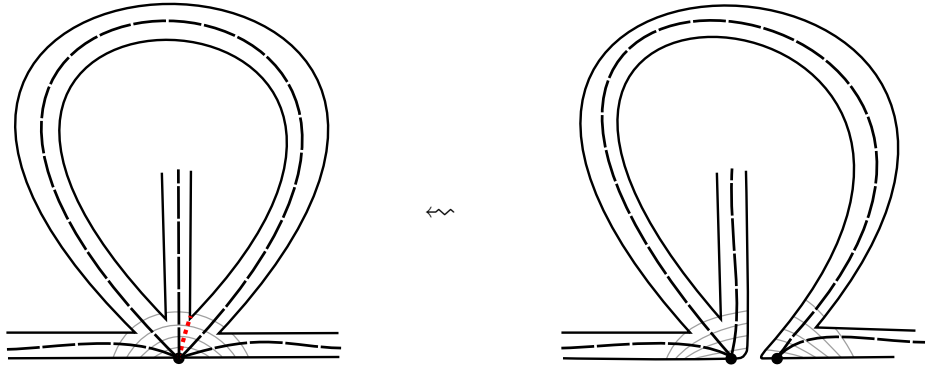}
	\caption{A loop in a skeleton of a foliated surface can be obtained via fusion.}
	\label{fig:FusionBVLoop}
\end{figure}

If a slide move contains a loop, we have that operators $\Delta^{\Gamma_\mathrm{unfused}}$
and $\Delta^{\Gamma'_\mathrm{unfused}}$ give the same $\Delta$  on the moduli space of the 
unfused surface, and the additional term from fusion does not depend on $\Gamma$.
\end{proof}
For more details on fusion of quasi-BV manifolds, see Appendix \ref{app:fusion}.

\newpage

\section[Topological interpretation of Delta]{Topological interpretation of $\Delta$}
 In this section, we give a topological interpretation of the operator $\Delta$, in terms of chords at intersections of loops on the surface.

\subsection{Curves and chords}
We start by introducing a class of functions on the moduli space, given
by evaluating functions on $G$ on holonomies along paths in $\Sigma$ and their derivatives using \emph{chords}.
Such functions appeared before in \cite{AMR}, we present a version
adapted to supergroups.

First, the case without chords is simply given by assigning holonomies to 
paths in $\Sigma$. For any path $\gamma$ on $\Sigma$ connecting two points
of $V$, we have a map $\hol_{\gamma}\colon \ModSpace{\Sigma, V}{G} \to G$ giving the holonomy
along $\gamma$. This map, by definition, depends only on the homotopy class of
  $\gamma$ fixing the endpoints.
\begin{definition}
	Let $\gamma_{1}, \dots, \gamma_{k}$ be $k$ paths between points of $V$.
	Then we define 
	$$\hol_{\gamma_{1}, \dots, \gamma_{k}} \colon \mathcal{O}(G^{\times k}) \to \mathcal {O} (\ModSpace{\Sigma, V}{G})$$
	as the pullback along the map $\ModSpace{\Sigma, V}{G} \to G^{\times k}$ given 
	by the product of maps $\hol_{\gamma_{i}}\colon \ModSpace{\Sigma, V}{G} \to G$.
\end{definition}
This map only depends on the homotopy class of the paths\footnote{By such homotopy
we mean a continuous  map $[0, 1]\times ([0, 1]^{\sqcup k}) \to \Sigma$
fixing the endpoints of the $k$ paths.} $\gamma_{i}$. 
If we choose a skeleton, with $N$ edges, we get a map 
$\mathcal O (G^{\times k}) \to \mathcal O (G^{\times N})$, as illustrated
on the following example.

\begin{example}
	If $\gamma$ is the boundary loop and $e_{1}, e_{2}$ are the two generators 
	of the fundamental group as on the Figure \ref{fig:TorusPath}, then we have $\gamma = e_{2}^{-1}e_{1}^{-1}e_{2}e_{1}$
	and thus the map $\hol_{\gamma}$ is given by the diagram on Figure \ref{fig:TorusPathDiag}.

\begin{figure}
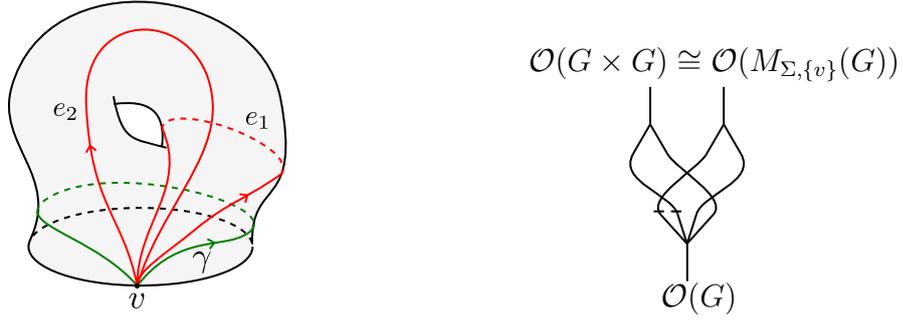

	\centering
	\begin{subfigure}[b]{0.4\textwidth}
		\begin{center}
			{\includesvg{images/torusboundary.svg}}
		\end{center}
		\caption{A punctured torus with a loop.}
		\label{fig:TorusPath}
	\end{subfigure}
	\qquad  
	\begin{subfigure}[b]{0.4\textwidth}
		\begin{center}
			{ \includesvg{images/torusboundarydiag.svg}}	
		\end{center}
		\caption{The corresponding map $\hol_{\gamma}$.}
		\label{fig:TorusPathDiag}
	\end{subfigure}
	\caption{Interpreting loops on a surface as functions on the moduli space.}
\end{figure}
	
	There are two ways to read this diagram, dual to each other.   
	From top to bottom, it is built from structure maps of a (super)group,
	i.e. the two strands correspond to the two copies of $G$, they are followed
	by two diagonal maps $G\to G\times G$ and the group multiplication in the correct order,
	with bars signifying inverses.
	\medskip
	
	From bottom to top, it can be seen as a diagram in the symmetric monoidal category of 
	vector superspaces.  It is built from
	structure maps of the (super) Hopf algebra $\mathcal O(G)$, i.e. an iterated coproduct,
	followed by the antipode (the bar), symmetry and the product of the Hopf algebra.
	Denoting the coproduct, antipode, symmetry and product by $\square, S, \tau$ and $m$, 
	this can be written as
	\[ (m \otimes m) \circ (1 \otimes \tau \otimes 1) \circ (\tau \otimes \tau) \circ (S^{\otimes 2}\otimes 1^{\otimes 2}) \circ \square^{(4)} \colon \mathcal O(G) \to \mathcal O(G) \otimes \mathcal O(G), \]
	where $\square^{(4)} =(\square \otimes 1^{\otimes 1}) \circ (\square \otimes 1) \circ \square$.	
\end{example}

Let us now assume that the Lie algebra $\g$ of $G$ is odd metric.
Then we can define functions assigned to paths with one chord, by acting 
with $t\in \g\otimes \g$ at the chord endpoints on the outgoing half-edges.

\begin{definition} \label{def:chord}
	Let $\gamma_{1}, \dots, \gamma_{k}$ be as before and let us choose a path $\delta$ on $\Sigma$
	connecting two points on $\gamma_{1} \sqcup \dots \sqcup \gamma_{k}$, 
	a so-called \emph{chord}.
	
	Then we define $\hol^{\delta}_{\gamma_{1}, \dots, \gamma_{k}}
	\colon \Pi\otimes\mathcal{O}( G^{\times k}) \to \mathcal {O} (\ModSpace{\Sigma, V}{G}) $
	as follows: First, we deform the paths so that the endpoints of $\delta$
	are in $V$, as on Figure \ref{fig:MoveEndpoint}
	
\begin{figure}[h]
	\centering
	\includesvg{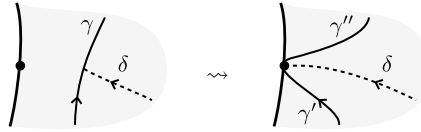}
	\caption{Moving endpoints of a chord to a marked point.}
	\label{fig:MoveEndpoint}
\end{figure}

	This gives us a map $\hol' \colon \ModSpace{\Sigma, V}{G} \to G^{\times(k+3)}$, 
	because of the additional holonomy along the chord and the two subdivisions.
	In the case when the chord connects $2$ different paths, 
	we define $\hol^{\delta}_{\gamma_{1}, \dots, \gamma_{k}}$ as on Figure \ref{fig:ChordDefOn2}.
	
\begin{figure}[h]
	\centering
	\includesvg{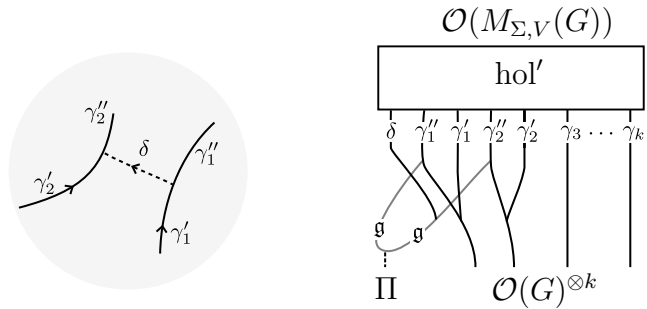}
	\caption{Definition of a chord connecting two paths.}
	\label{fig:ChordDefOn2}
\end{figure}

	If the chord lies one path, we define $\hol^{\delta}_{\gamma_{1}, \dots, \gamma_{k}}$ as on Figure \ref{fig:ChordDefOn1}.
	Finally, we require that changing the orientation of the chord is equivalent to multiplication by $-1$.

\begin{figure}[h]
	\centering
	\includesvg{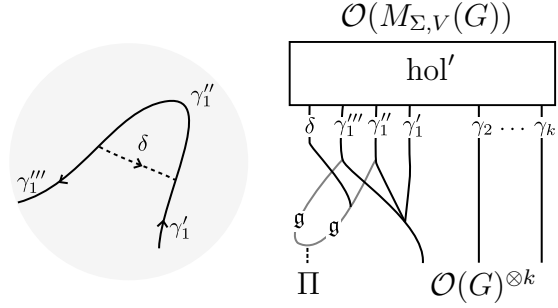}
	\caption{Definition of a chord connecting two points on one path. Oppositely oriented chord would result in $-1$ times the above diagram.}
	\label{fig:ChordDefOn1}
\end{figure}	
	
	The symbol 	\raisebox{-0.32\height}{\includesvg{images/symbolt.svg}} is the odd 
	tensor $(-1)^{\hdeg{e_{i}}}t^{ij} e_{i}\otimes e_{j} \colon \Pi \to \g \ot \g$. 
	We use the adjoint action $\g \to   \mathcal O(G)\otimes \g$ to (co)act by
	the holonomy along $\delta$. The Lie algebra action $\g \otimes \mathcal O(G)\to \mathcal O(G)$ is by left-invariant vector 
	fields.
\end{definition}

A useful rule of thumb is that we act on the half-edge of $\gamma$ leaving the chord endpoint. Intuitively, one of the endpoints of the chord is acted on by the holonomy along the chord. Specifically, note that the chord is a path lying on $\Sigma$, and chords 	with different homotopy classes will act differently.

\begin{remark}
	In the case of usual Lie algebra, one would write the first map, evaluated at $f_{1}\otimes f_{2}\in \mathcal O (G\times G)$, as
	\[ \sum_{i, j} s^{ij} \left. \frac{\partial}{\partial t_{1}}\right | _{t_{1} = 0} \hspace*{-5mm}f_{1}\left(\gamma''_{1}e^{t_{1} e_{i}} \gamma'_{1}\right) 
	\, \left.\frac{\partial}{\partial t_{2}}\right|_{t_{2} = 0} \hspace*{-5mm} f_{2}\left(\gamma''_{2}e^{t_{2} \mathrm{Ad}_{\delta}(e_{j})} \gamma'_{2}\right).\]
	where we use $\delta$ or  $\gamma$  to denote the holonomies along the respective loops
	\medskip
	
	In the graded case, the diagram should be read from the bottom to top, as a diagram in the category of supervector spaces. The grey lines represent the Lie algebra $\g$, the black lines the super Hopf algebra $\mathcal O(G)$, and the dotted line represents $\Pi = \mathbb R^{1|0}$. The adjoint action $\g \to\mathcal O(G)\otimes \g$ would correspond in the even case to the map sending $x\in \g$ to the $\g$-valued function $g\mapsto \Ad_g x$.
	
\end{remark}

\begin{proposition}
	The map $\hol^{\delta}_{\gamma_{1}, \dots, \gamma_{k}}$ only depends on the homotopy
	classes of the curves (the endpoints of the chord only move along the path they belong to),
	and is independent of how we deformed the endpoints of the chord to $V$.
	Changing the orientation of the chord multiplies the corresponding $\hol$
	function by $-1$.
\end{proposition}
\begin{proof}
	Choosing a different path for a chord endpoint, possibly
	ending at a different marked vertex, multiplies the holonomy along $\delta$ by a 
	well-defined holonomy $g\in G$ (computed from the new vertex to the old vertex). 
	Let us treat only the case when the starting point of the chord is moved,
	since the case of the endpoint follows by reversing the direction. The 
	new holonomy along the moved chord is given by $g$ followed by the holonomy along
	the old chord, i.e. $\delta \mapsto \delta g$. The two holonomies
	$\gamma'$, $\gamma''$  change to $g^{-1}\gamma'$ and $\gamma'' g$.
	All of these actions cancel out due to the invariance of $t$, i.e. 
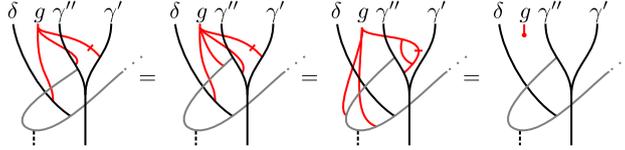
	\begin{center}
		\includesvg{images/proof_indep.svg}
	\end{center}
	where the additional holonomy $g$ is highlighted in red.

	The dependence on orientation of the chord follows from the invariance
	and antisymmetry of $t$.  
\end{proof}


Moreover, we can get a function on $\ModSpace{\Sigma, V}{G}$ from a loop in $\Sigma$, 
provided that it is assigned a function invariant under conjugation by $G$.
\begin{proposition}
	Restrict $\hol^{\delta}_{\gamma_{1}, \dots, \gamma_{k}}$ to the subspace of functions 
	on $G^{\times k}$ which are invariant under conjugation of the $i$th factor. 
	Then this map depend only on the free homotopy class of $\gamma_{i}$,
	i.e. it is independent on $v\in V$ and of 
	the representative in conjugacy classes in $\pi_{1}(\Sigma, v)$.
\end{proposition}
\begin{proof}
	Changing how the loop $\gamma_{i}$ is deformed to a based loop (also possibly changing
	the base point from $v$ to $v'$) conjugates the holonomy, under which the function $f$ is invariant. 
	The case when there is a chord on $\gamma_{i}$ is analogous; the case when both
	endpoints lie on $\gamma_{i}$ is slightly more involved, and follows from a 
	suitable version of the equation
	\[   f_{i}( d\, e^{te_{i}} c\,b\, e^{t'\Ad_{\delta}(e_{j})} a ) = f_{i}( b \,e^{t' \Ad_{\delta}(e_{j})}  a \,d\, e^{t e_{i}} c  ) \]
	with the holonomies $a, b, c, d$ as on Figure \ref{fig:ConjInv}.
	\begin{figure}[h]
		\centering
			\includesvg{images/conjinv.svg}
		\caption{Notation for holonomies for conjugation invariance, in the case of changing a basepoint of $\gamma_i$.} 
		\label{fig:ConjInv}
	\end{figure}
\end{proof}

\subsection{Quasi-BV structure on holonomies}
We can now express the action of the BV operator on these functions coming 
from chords. We will from now on assume that all the families of paths have a finite
number of transverse double (possibly self-) intersections and no other intersections.
We also assume that the curves intersect the boundary only at their endpoints, which lie in
$V$, and that they become tangent to the foliation in a finite number of points only.

\begin{definition}\label{def:deltaintersections}
	For any collection $\gamma_{1}, \dots, \gamma_{n}$ of paths on $\Sigma$, let $p$ be an intersection
	or self-intersection point. Let us place a chord connecting 
	the two segments near $p$ as on Figure \ref{fig:FourChIn} for an intersection in the interior of $\Sigma$ and as on Figure \ref{fig:FourChBnd} for
	an intersection at the boundary. 
	
\begin{figure}[h]
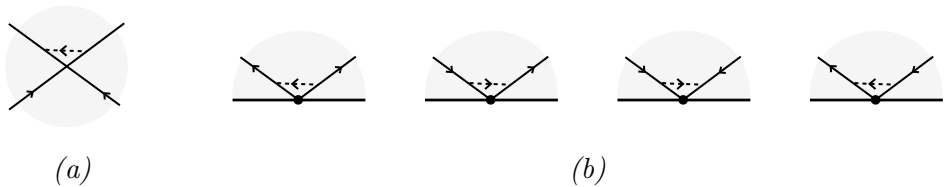

	\centering
	\begin{subfigure}[b]{0.14\textwidth}
		\begin{center}
			\includesvg{images/fourchordsv2.svg}
			\vspace{-3mm}	
	\end{center}
		\caption{}
		\label{fig:FourChIn}
	\end{subfigure}
	\quad  
	\begin{subfigure}[b]{0.7\textwidth}
		\begin{center}
			\includesvg{images/fourchordsboundary.svg}
		\end{center}
		\caption{}
		\label{fig:FourChBnd}
	\end{subfigure}
	\caption{Rules for adding chords at intersections in the interior and at the boundary. The surface orientation is counter-clockwise.}\label{fig:FourCh}
\end{figure}	
	
	Define a function on the moduli space given by 
	the sum over all intersection points
	\begin{equation} \label{eq:chordsfromintersections}
		 \sum_{p} \lambda_{p} \; \hol^{\delta_{p}}_{\gamma_{1}, \dots, \gamma_{k}} + \frac 12 \hol_{\gamma_{1}, \dots, \gamma_{k}} \circ \sum_{k} \mathrm{rot}_{\gamma_{i}} \, \nu^{\mathrm L, i},
	\end{equation}
	where $\lambda_{p}$ is $1$ if the intersection is in the interior of the surface
	and $1/2$ if the intersection is at the boundary. The second sum is 
	over all paths, acting with the vector $\nu\in \g$ as a left-invariant vector field on the $i$th copy
	of $G$ in $G^{\times k}$.
\end{definition}
The second sum can be seen as adding a $\mathrm{rot}_{\gamma_{i}}$-multiple
of a short chord on each path $\gamma$, as we will see it the proof of the following proposition.

\begin{proposition}\label{prop:reidemeister}
	The above function does not change if we move the individual paths
	$\gamma_{i}$ by homotopy. 
\end{proposition}
\begin{proof}
	We need to check invariance with respect to the three Reidemeister moves (see \cite[{Proof~of~Prop.~2.11}]{Nie2013}).
	
	\textbf{RI:}
		Adding a twist to $\gamma_i$ adds two terms to \eqref{eq:chordsfromintersections}, from the new self-intersection and from the change of the rotation number.  At the intersection, we get a \emph{short chord} oriented as follows:
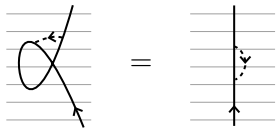
		\begin{center}
		\includesvg{images/shortchord.svg}
		\end{center}
        From Definition \ref{def:chord}, such chord acts by 
		$ (-1)^{\hdeg{e_{i}}}t^{ij}(e_{i})^{\mathrm L}(e_{j})^{\mathrm L} = \nu^{\mathrm L}/2$,
		see also Proposition \ref{prop:tproperties}.
		This cancels with $-\nu^{L}/2$ coming from the full additional clockwise turn the path undertakes.

	\textbf{RII:}
		We get a cancellation
		\begin{center}
			\includesvg{images/RII.svg}
		\end{center}
		which holds by the invariance and antisymmetry of the chord. Other possible
		cases follow since changing the orientation of any 
		of the paths multiplies the whole equation by $-1$. If one of the
		intersection points is at the boundary, the situation is analogous.
			
	\textbf{RIII:} 	The following identity holds term by term. Again, changing the orientation of any 
		of the paths multiplies the terms where a chord lies on that path by $-1$.
		\begin{center}
			\includesvg{images/RIII.svg}
		\end{center}
	
\end{proof}

\begin{theorem} \label{thm:geometricBV}
	Let $\gamma_{1}, \dots, \gamma_{k}$ be as before. Then
	\begin{equation}\label{eq:BVintersection}
		\Delta \circ \hol_{\gamma_{1}, \dots, \gamma_{k}} = 
	 \sum_{p} \lambda_{p} \; \hol^{\delta_{p}}_{\gamma_{1}, \dots, \gamma_{k}} +  \frac 12 \hol_{\gamma_{1}, \dots, \gamma_{k}} \sum_{i} \mathrm{rot}_{\gamma_{i}} \nu^{\mathrm L, i}.
	\end{equation}
	where the right hand side is defined in Definition \ref{def:deltaintersections}
\end{theorem}

The formula \eqref{eq:BVintersection} completely determines
the quasi BV operator, since we can get any function on the moduli space via holonomies.
\begin{corollary}
	The quasi BV operator acting on $\ModSpace{\Sigma, V}{G}/G_{v}$ is equal to the quasi BV operator
	on $\ModSpace{\Sigma, V\setminus \{v\} }{G}$. Specifically, there is a canonical BV operator
	on $\mathcal O(\ModSpace{\Sigma}{G})$, i.e. on the $G^{\times V}$-invariant functions in $\mathcal O(\ModSpace{\Sigma, V}{G})$. The choice
	of $V$ is arbitrary. In general the foliation must be deformed to be compatible with
	$V$, which is always possible, see Figure \ref{fig:DeformFoliation}.
\begin{figure}[h]
	\centering
		\includesvg{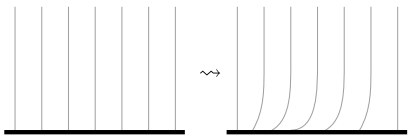}
	\caption{Deforming a foliation to be compatible with a new marked point.}
	\label{fig:DeformFoliation}
\end{figure}
\end{corollary}
\begin{proof}
The first claim follows since for both quasi BV operators on  $\ModSpace{\Sigma, V}{G}/G_{v}$ and on $\ModSpace{\Sigma, V\setminus \{v\} }{G}$, the RHS of \eqref{eq:BVintersection} is the same. Similarly, the quasi BV operator on $\mathcal O(\ModSpace{\Sigma}{G})$ can again be computed using the RHS of \eqref{eq:BVintersection}. 
\end{proof}

\begin{proof}[Proof of Theorem \ref{thm:geometricBV}]
	Both sides of Equation \eqref{eq:BVintersection} are maps $$\mathcal O(G^k) \to \mathcal O(\ModSpace{\Sigma, V}{G}).$$ The left hand side, evaluated on a tensor product of $k$ functions on $G$, is equal to the BV operator from Theorem \ref{thm:qBVFR} acting on the function on the moduli space given by evaluating these $k$ functions on the $k$ holonomies along $\gamma_1, \dots, \gamma_k$. 
		
	The right hand side is equal to a similar function given by the holonomies, \emph{together with one chord for each intersection of $\gamma_i$} (plus the term containing rotations of $\gamma_i$.)
	
	 Let us describe the strategy of the proof. We will start with the left hand side of \eqref{eq:BVintersection}. We have an explicit formula \eqref{eq:qbvdef} for the quasi-BV operator, once we choose a skeleton $\Gamma$. If we deform the paths $\gamma_i$ to intersect only near the vertices of $\Gamma$ (the marked points), the action of the first term of \eqref{eq:qbvdef} can be rewritten in terms of chords from Definition \ref{def:chord}. These chords will act on any pair of path segments meeting at a marked point. Analyzing the possible positions of such pairs of path segments, we will show that if the paths don't intersect, these chords will cancel each other; if the paths do intersect, we recover the intersection formula \eqref{eq:chordsfromintersections}.
	\medskip
	
	Let us start with $\Delta \circ \hol_{\gamma_{1}, \dots, \gamma_{k}}$.
	For a chosen skeleton $\Gamma$, $\hol_{\gamma_{1}, \dots, \gamma_{k}}$ is 
	given (from bottom to top) by 
	applying the iterated coproduct on each component of $\mathcal O(G^{\times k})$,
	then a permutation of these factors, and finally by multiplying
	together factors corresponding to the same edge of $\Gamma$.

	Both sides of equation \eqref{eq:BVintersection} are homotopy invariant.	
	Thus, the above morphism can be visualized by retracting the surface, and with it 
	the paths $\gamma$, to $\Gamma$. Let us deform the paths $\gamma$ such that they only intersect in small neighborhoods of vertices of $V$, as on Figure \ref{fig:Highway}. For each $v\in V$, we will call the connected components of the intersection of this neighborhood with $\cup_i \gamma_i$ \emph{segments}. By possibly further changing $\gamma$ by homotopy, we can ensure that each pair of segments intersects at most once in this neighborhood, as on Figure \ref{fig:Highway}. This can be achieved, in generic case, by e.g. making the segments straight.
\begin{figure}[h]
	\centering
		\includesvg{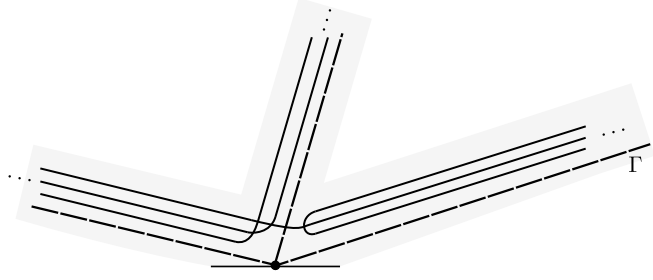}
	\caption{Paths intersecting only near a marked point $v$.}
	\label{fig:Highway}
\end{figure}	
	
	The operator $\Delta$ in $\Delta \circ \hol_{\gamma_{1}, \dots, \gamma_{k}}$ 
	(defined in \eqref{eq:qbvdef}, ignoring the rotation term and signs for a moment) acts on pairs of 
	half-edges of $\Gamma$. By Leibniz rule, we get at each vertex a sum 
	$$\sum_{i < j, i, j \text{ not in the same half-edge of $\Gamma$}} \frac12 \hol^{i\to j}_{\gamma_{1}, \dots, \gamma_{k}}$$ 
	over all half-edges of the part of $\gamma_{1}, \dots, \gamma_{k}$ incident to the vertex.
	The chord $i\to j$ connects the two half-edges $i$ and $j$ and comes 
	from the term $\tilde{t}_{ab}$ of $\Delta$.{}
	
	We can remove the condition that $i$ and $j$ do not follow the same half-edge of $\Gamma$. Indeed, the above sum is equal to
	$$\sum_{i < j}  \hol^{i\to j}_{\gamma_{1}, \dots, \gamma_{k}},$$ 
	where the added terms cancel with terms from neighboring vertices by the invariance
	and antisymmetry of $t$.
	
	The sum above contains chords connecting half-edges of paths close to marked points.
	Every chord connects either two consecutive half-edges in
	a path (i.e. lies on one path segment), or connects two different segments of (possibly the same) path going through the marked point.
	Splitting the above sum, we get
	\begin{equation}\label{eq:sumofchords}
	\sum_{e_{1} \neq e_{2}} \frac 12 (\text{1 to 4 chords between these two segments} ) + \frac 12 \sum_{e}\hol^{e}_{\gamma_{1}, \dots, \gamma_{k}}.
	\end{equation}
	
	The first term is a sum over all pairs of segments $(e_1, e_2)$ of paths $\gamma_{i}$ going through $v$, and for each such pair, 
	we collect all the chords that connect them. In the second term,
	we get a ``short'' chord placed on the segment $e$
	(at the marked point), which equals the action of $\pm\nu/2$. Together with the rotation part 	of $\Delta$, it combines to give the second term of the RHS of \eqref{eq:BVintersection},
	since at each vertex the path undergoes an extra $1/2$ turn in addition
	to the rotation along edges of $\Gamma$.
	\medskip{}
	
	\textbf{Signs:} If a half-edge $\gamma'$ arrives at the vertex $v$, the 
	element $x\in\g$ coming from the chord $\tilde{t}_{ab}$ acts as $x^{-\mathrm{R}}$ on that holonomy;
	for an outgoing edge $\gamma''$, $x$  acts by $x^{\mathrm{L}}$, see the Figure \ref{fig:convention}. However, 
	for a function of the product $\gamma''\gamma'$, the equality $x^{L''} = x^{R'}$ holds. We can thus move
	every action to act by left invariant vector fields on the outgoing half-edge,
	as in Definition \ref{def:chord}, with a minus sign for each chord acting on an incoming edge. 
	Let us now apply this rule.
	\medskip
	
	In the case where the chord lies on one path segment, we get a $\tfrac 12$ times the short chord, as picture on Figure \ref{fig:PC0}
\begin{figure}[h]
	\centering
	\includesvg{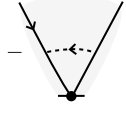}
	\caption{A chord lying on one path segment from the second sum in \eqref{eq:sumofchords}.}
	\label{fig:PC0}
\end{figure}	The orientation of the chord is given by the order of half-edges, as in equation 
	\eqref{eq:qbvdef}, with the minus sign as discussed above. This gives an action of $-\nu/4$, which is consistent with the counter-clockwise half-turn, as on
	Figure \ref{fig:foliation}.

	\medskip

\begin{figure}[h]
	\centering
	\includesvg{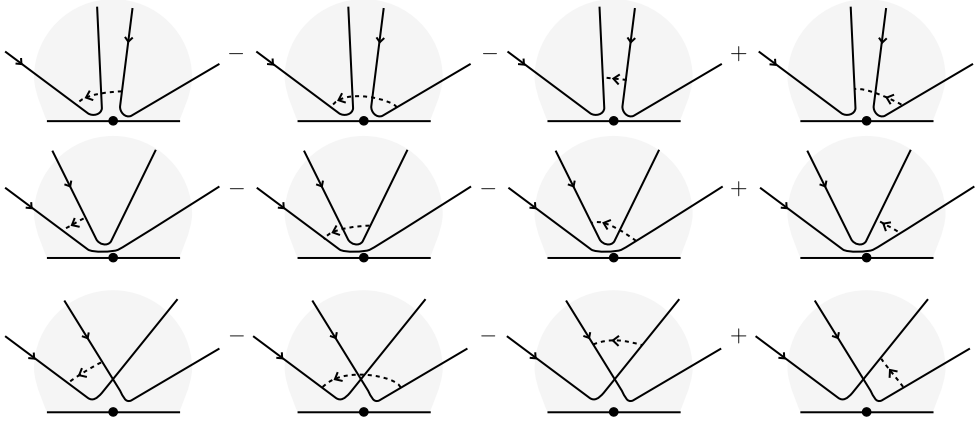}
	\caption{Case with two traversing paths.}
	\label{fig:PC1}
\end{figure}

	Next, we treat the sums over pairs of segments from Equation \eqref{eq:sumofchords} case by case. If two path segments meet near the marked point $v$, none, one or both of them start on end at $v$. In all cases, we will show that we recover the intersection rule from Definition \ref{def:deltaintersections}:
	\begin{enumerate}
		\item If neither of the paths starts nor ends at the marked point,
		there are $4$ chords. There are three different ways to 
		connect $4$ half-edges to $2$ path segments, two in which the paths don't intersect and one in which they do. These
		four possibilities are shown on Figure \ref{fig:PC1}.

		The first two lines, without an intersection of the segments, vanish. 
		In the last line, the first two terms cancel,
		but the other two chords add up. This corresponds to the chord 
		that comes from two paths intersecting from Definition \ref{def:deltaintersections}.
		The direction is correct (chord leaves the first outgoing half-edge)
		and the sum of two chords cancels with the factor $\tfrac 12$ in Equation \eqref{eq:sumofchords}.
		
		\item If one path starts or ends at $v$, we get the three
		cases shown on Figure \ref{fig:PC2}. 		As before, we get a non-zero contribution only in the last case, 
		which, after multiplying by $\tfrac 12$ from \eqref{eq:sumofchords}, has the correct factor of $+1$.
\begin{figure}[h]
	\centering
	\includesvg{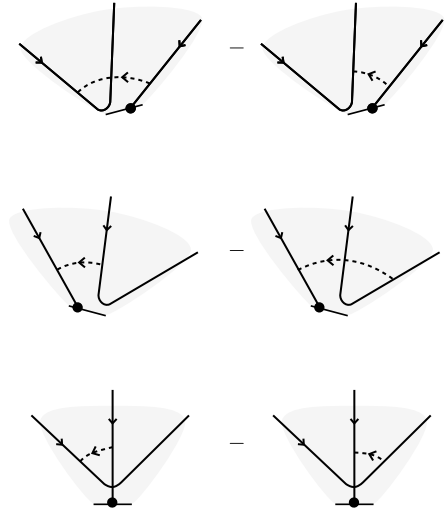}
	\caption{Case with one path ending at the marked point.}
	\label{fig:PC2}
\end{figure}
		\item If both paths start or end at $v$, there is only one term, shown on Figure \ref{fig:PC3}. Together with the factor $\tfrac12$ from \eqref{eq:sumofchords}, we get an agreement with Definition \ref{def:deltaintersections}.
\begin{figure}[h]
	\centering
	\includesvg{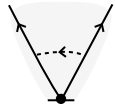}
	\caption{Case with both paths ending at the marked point.}
	\label{fig:PC3}
\end{figure}

	\end{enumerate}
	The remaining cases, in which the paths are oriented differently, 
	follow from the above calculations: changing an orientation of one of the
	paths changes signs on both sides of the equation.
\end{proof}

\subsection{Another formula for the quasi-BV operator}
We will now present two more formulas for the quasi-BV operator in terms of
the surface. Their role is to make the role of the foliation clearer. 
Concretely, in the formula \eqref{eq:chordsfromintersections}, the term containing the rotation numbers $\mathrm{rot}_\gamma$ has to be added by hand to the sum over all intersections. If we instead consider intersection of the collection of paths $\gamma_1, \dots, \gamma_k$ and \emph{the same collection, shifted in the direction of the foliation}, we will obtain the rotation numbers automatically. We will first consider a case of a general foliation, and then a simpler situation of an orientable foliation (see Appendix \ref{app:foliation}).

We will shift the paths $\gamma_{i}$ in the direction of the foliation.
For each path $\gamma_i$, there are two possible choices for the direction of 
this shift. 

\begin{definition}\label{def:shifts1}
	 Let $\gamma_{1}, \dots, \gamma_{k}$ as before.
	 Choose a direction of a small shift for each $\gamma_{i}$ such that the shifted and unshifted paths all intersect transversally in double points. If
	 a path $\gamma_{i}$ intersects a shifted path $\gamma_{j}^{\mathrm{sh}}$
	 in $p$, the point $p$ has a preimage on the original path $\gamma_{j}$,
	 let us call it $p_{0}$.  See Figure \ref{fig:ChordSh}, with the shifted path shown in red.
	
\begin{figure}[h]
	\centering
		\includesvg{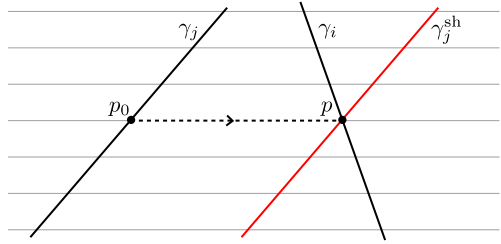}
	\caption{Rule for chords at the intersection of a shifted (red) and an unshifted (black) path}
	\label{fig:ChordSh}
\end{figure}
	 
	 Then, we define a function on the moduli space by averaging over all
	 possible choices of the directions of shifts, and for each choice
	 by summing over all intersections of shifted and unshifted paths,
	 with a chord going from $p_{0}$ to $p$
	 \begin{equation}\label{eq:shifts1}
		 \frac{1}{2^{k}}\sum_{\text{$2^{k}$ possible shifts}}  \quad
		 \sum_{p \in \gamma\cap\gamma^{\mathrm{sh}}} \frac 12 \alpha_{p}\,\hol_{\gamma_{1}, \dots, \gamma_{k}}^{p_{0} \to p}.
	 \end{equation}
	 where the sign $\alpha_{p}$ is $+1$ iff the half-edges leaving $p$,
	 ordered (shifted, unshifted), are compatible with the orientation of 
	 the surface. 
\end{definition} 


\begin{proposition}\label{prop:shifts1}
	The function \eqref{eq:shifts1} from the above definition
	is equal to $\Delta\circ \hol_{\gamma_{1}, \dots, \gamma_{k}}$.
\end{proposition}
\begin{proof}
	An intersection point $p$ occurs either where two path segments intersect,
	or when the path becomes tangent to the foliation. For brevity, let us
	denote the choice of directions of the shifts by $C$.
	\medskip{}
	
	In the first case, if the intersection happens away from the boundary, 
	we get (see Figure \ref{fig:SP1}), for each $C$, two chords at the intersection, with all the 2 or 4 possible
	shifts giving the same answer.
	
\begin{figure}[h]
	\centering
	\includesvg{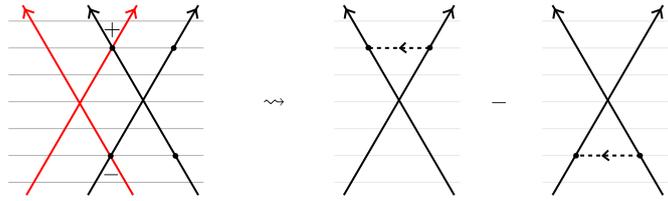}.
	\caption{Definition \ref{def:shifts1} applied near an intersection in the interior of $\Sigma$.}
	\label{fig:SP1}
\end{figure}
	
	Together with the factor $\tfrac12$ from Equation \eqref{eq:shifts1},  we get an agreement with Definition \ref{def:deltaintersections}.
	\smallskip{}
	
	If the intersection happens on the boundary, it is either
	 an intersection of two different paths, or a self-intersection.
	In the first case, the four cases on Figure \ref{fig:SP2}
	appear for different $C$.
	
\begin{figure}[h]
	\centering
	\includesvg{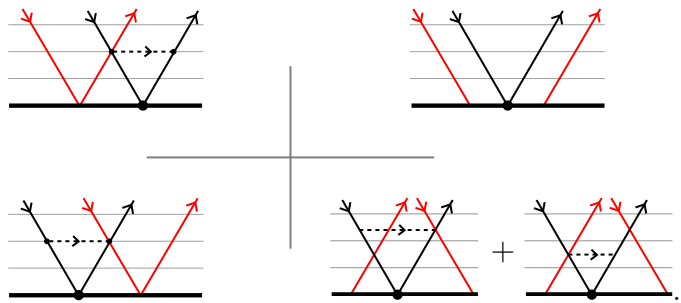}.
	\caption{Definition \ref{def:shifts1} applied near an intersection at the boundary of $\Sigma$.}
	\label{fig:SP2}
\end{figure}
	Here, we already oriented the chord  to absorb the possible sign
	$\alpha_{p}$; the terms also carry a factor $\tfrac12$.
	
	If the segments meeting at the boundary belong to the same path, there are two cases to distinguish:
    If the foliation makes an odd number of turns along the path, we get the 
	left column of the above figure, and for an even number of turns,
	we get the right column of the above figure.
	\smallskip
	
	Finally, if the path becomes tangent to the foliation, there is a contribution only
	in the cases if it looks like a local extremum, see Figure \ref{fig:SP3}.
	
\begin{figure}[h]
	\centering
	\includesvg{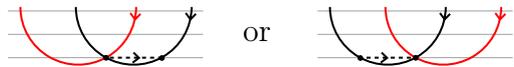}.
	\caption{Definition \ref{def:shifts1} when a path becomes tangent to the foliation.}
	\label{fig:SP3}
\end{figure}

	This short chord, together with the factor $\tfrac 12$, acts by $\nu/4$, as expected (compare with the proof of Proposition \ref{prop:reidemeister})
		\smallskip{}
		
	Since all of these cases have the same frequency among all possible $C$,
	we see that the average number of chords is $\tfrac 12$ in all three cases, which recovers the factor $\lambda_p$ from 
	Definition \ref{def:deltaintersections}.
	As before, the cases with different orientations of the intersecting segments
	follow from these, since changing an orientation multiplies the term 
	by $-1$.
\end{proof}

However, if the foliation is orientable\footnote{See Appendix \ref{app:foliation}.}, i.e. we can consistently choose a direction of the shift, we can remove the symmetrization from 
Equation \eqref{eq:shifts1}.
\begin{definition}
	Let $\gamma_{1}, \dots, \gamma_{k}$ be as in Definition \ref{def:shifts1}
	and assume that the foliation of $\Sigma$ is orientable. Choose one
	such orientation, shift all paths $\gamma$ along this vector field and define,
	as before, a function by summing over all intersections of the shifted and unshifted paths.
	\begin{equation}\label{eq:shifts2}
		\sum_{p \in \gamma\cap\gamma^{\mathrm{sh}}} \frac 12 \alpha_{p}\hol_{\gamma_{1}, \dots, \gamma_{k}}^{p_{0} \to p}.
	\end{equation}
	The sign $\alpha_{p}$ and the chord direction is as in Definition \ref{def:shifts1}.
\end{definition}
\begin{proposition}
	The function \eqref{eq:shifts2} from the above definition
	is equal to $\Delta\circ \hol_{\gamma_{1}, \dots, \gamma_{k}}$.
	Specifically, it does not depend on the sign of the orientation of the foliation. 
\end{proposition}
\begin{proof}
	The proof is similar to the proof of Proposition \ref{prop:shifts1},
	only the foliation is now oriented, which excludes the cases in right column,
	when the intersection happens at the boundary. The remaining cases
	recover Definition \ref{def:deltaintersections} without averaging.
\end{proof}
\section[Goldman-Turaev Lie bialgebra and Q(n)]{Goldman-Turaev Lie bialgebra and $Q(n)$}
In this section, we specialize to $G = Q(n)$, 
the queer Lie supergroup. This will allow us to relate the Goldman-Turaev Lie bialgebra with the BV operator on the moduli space of flat connections, extending the correspondence of the 
Goldman bracket and the Atiyah-Bott Poisson structure \cite{Goldman}.

\subsection[The Lie supergroup Q(n)]{The Lie supergroup $Q(n)$}\label{ssec:queerGT}
Let us recall the definition of the queer Lie supergroup $Q(n)$ (see \cite[{\S1.8}]{Leites2011seminar} for more details). 
\begin{definition}
	For $n \ge 1$, define the following associative algebra
	\[  q_{\mathrm{as}}(n) = \operatorname{Mat}_n(\mathbb R) \otimes \mathbb R[\xi] /(\xi^{2} - 1). \]
	This algebra is $\mathbb Z_{2}$-graded by setting $\xi$ to be odd. 
	The odd function $\mathrm{otr}$ on $q_{\mathrm{as}}(n)$ is defined by \[\mathrm{otr} (X + \xi Y) = \tr Y\]
	and is cyclically symmetric\footnote{The usual Koszul sign $(-1)^{\hdeg A_{1} \hdeg A_{2}}$
	is equal to $+1$, since $A_{1}$ and $A_{2}$ have opposite parity.}
	 $\mathrm{otr}(A_{1} A_{2}) = \mathrm{otr}(A_{2}A_{1})$.
	 
	We define $Q(n)$, to be the Lie supergroup of invertible elements associated to $q_{\mathrm{as}}(n)$. Its Lie superalgebra, denoted $\q(n)$,
	is the space $q_{\mathrm{as}}(n)$  with the bracket given by the graded commutator. The odd trace makes $\q(n)$ into an
	odd metric Lie algebra, with pairing given by $A_{1}\otimes A_{2}  \mapsto \mathrm{otr}(A_{1}A_{2})$.
\end{definition} 
The Lie superalgebra $\q(n)$ is unimodular, see Proposition \ref{prop:queerisdiscardy} and Remark \ref{rmk:DiscardyUnimodular}.

\subsection{Goldman-Turaev Lie bialgebra}
Any collection of $k$ loops $\gamma_{i}$ on $\Sigma$ gives a function $\hol_{\gamma_1, \dots, \gamma_k}(\otr \otimes \dots \otimes \otr)$ on the moduli space
$\ModSpace{\Sigma, V}{Q(n)}$, by taking a product of the odd traces of holonomies along $\gamma_{i}$. Our goal is now
to study the action of the BV operator $\Delta$ on such functions, to which end
we need to recall the Goldman-Turaev Lie bialgebra.

Recall from Section \ref{ssec:Goldman} that $\Goldman{\Sigma} = \mathbb{R} \pi_{1}^{\mathrm{free}}$ is the 
$\mathbb R$-vector space  generated by homotopy classes of free loops in an oriented surface $\Sigma$. The following two operations were defined by 
Goldman and Turaev \cite{Goldman, Turaev}. 
\begin{definition}\label{def:GoldmanTuraev} Let $\gamma_{1}$, $\gamma_{2}$ be two immersed loops on $\Sigma$ representing
their classes $|\gamma_{1}|, |\gamma_{2}|\in \Goldman{\Sigma}$ with transversal double intersections.
Their Goldman bracket is given by a sum over their intersections
\[ [  |\gamma_{1}| , |\gamma_{2}| ]_{\mathrm G} =  \sum_{p \in \gamma_{1}\cap \gamma_{2}} \beta_{p} \Bigg[\; \raisebox{-18pt}{\includesvg{images/goldman.svg}} \;\Bigg], \]
where we modify the loops only in a small disc around $p$, connecting them into one loop.
The sign $\beta_{p}$ is $+1$ iff the two tangent vectors $(\dot{\gamma}_1, \dot{\gamma}_2)$ at $p$ agree with the orientation of $\Sigma$.

The Turaev cobracket of $|\gamma_{1}|$ is defined as a sum over all self-intersections of $\gamma_{1}$
\[ \delta_{\mathrm T} |\gamma_{1}| = \sum_{p \text{ self-intersection of }\gamma_{1}} \raisebox{-13pt}{\includesvg{images/goldman.svg}}, \]
where we see the resulting two loops as lying in $\Goldman{\Sigma} \wedge \Goldman{\Sigma}$, with the first 
loop being the one starting to the right (this is fixed by the orientation of the surface).
\end{definition}
Goldman proved \cite{Goldman} that $(\Goldman{\Sigma}, [\cdot, \cdot]_\textnormal{G})$ is a well-defined Lie algebra. Moreover, the constant loop is in the center of 
$[\cdot,\cdot]_{\mathrm{G}}$ and on the quotient $\Goldmanred{\Sigma} = \Goldman{\Sigma}/\mathbb R \bigcirc$, the above bracket and cobracket give a well-defined Lie bialgebra by a result of Turaev \cite{Turaev}. Moreover, $(\Goldmanred{\Sigma}, [\cdot , \cdot ]_{\mathrm G}, \delta_{\mathrm T})$  is involutive, i.e. $[\cdot , \cdot ]_{\mathrm G} \circ \delta_{\mathrm T} = 0$, by a result of Chas \cite{Chas2004}. 
Therefore, one can define a BV algebra structure\footnote{There are two conventions for a definition of a BV algebra used in literature, either with $\Delta(xy) = \Delta(x)y + (-1)^{\hdeg{x}} x \Delta(y) + \{x, y\}$ or with $\Delta(xy) = \Delta(x)y + (-1)^{\hdeg{x}} x \Delta(y) + (-1)^{\hdeg{x}}\{x, y\}$. The bracket $\{\cdot, \cdot\}$ is then either graded-symmetric or satisfies $\{y, x\} = (-1)^{(\hdeg{x}+1)(\hdeg{y}+1)+1} \{x, y\}$, respectively. We use the first convention.} on the commutative superalgebra $\wedge \Goldmanred{\Sigma}$ as in Theorem \ref{thm:CEintro} {\cite[Sec.~5]{CMW2016}}. The BV operator $\Delta^{[\cdot,\cdot]_\mathrm G, \delta_\mathrm T}$ is explicitly given by 
\begin{equation}\label{eq:CEBV}
\begin{aligned}
\Delta^{[\cdot,\cdot]_\mathrm G, \delta_\mathrm T} (x_1, \dots, x_n) = 
	 &\sum_{i<j} (-1)^{i+j+1} [x_{i},x_{j}]_\mathrm G x_{1}  \dots \hat{x}_{i} \dots \hat{x}_{j} \dots x_{n}
	\\ + &\sum_{i} (-1)^{i-1}x_{1}  \dots  \delta_{\mathrm T}({x_{i}})  \dots  x_{n},\end{aligned}
\end{equation}
where $x_i \in \Goldmanred{\Sigma}$ and we omit the symbol $\wedge$.

	

\subsection{The odd Goldman map}
As we mentioned above, collection of \emph{loops} $\gamma_{1}, \dots, \gamma_{k}$ on $\Sigma${}
defines a function on $\ModSpace{\Sigma, V}{Q(n)}$ by taking the product of odd traces of all the holonomies.
Since this function depends on the order of the loops $\gamma_{i}$ only up to the sign 
of a permutation, we get a map $\ogm \colon \wedge \Goldmanred{\Sigma} \to \mathcal O( \ModSpace{\Sigma, V}{Q(n)})$.   We will now show that this map intertwines the natural BV operators on both sides, defined in \eqref{eq:CEBV} and Theorem \ref{thm:qBVFR}, respectively. 

\begin{theorem}\label{thm:holonomyisBVmap}
	Let $\ogm \colon \wedge \Goldmanred{\Sigma} \to \mathcal O( \ModSpace{\Sigma, V}{Q(n)})$ be the algebra
	map defined by sending the generators $\gamma$ to $\hol_{\gamma}(\otr)$. Then
	\begin{equation}\label{eq:PhiOddMorph}   \Delta \circ \ogm = \ogm \circ \Delta^{[\cdot,\cdot]_{\mathrm G}, 2\delta_{\mathrm T}} ,\end{equation}
	i.e. $\ogm$ is a map of (quasi-)BV algebras.
\end{theorem}
Note that in order to get an agreement, we need to use $2\delta_{\mathrm T}$ as a cobracket on $\Goldmanred{\Sigma}$ in \eqref{eq:CEBV}.

\begin{proof}

	It will be simpler to consider, instead of the supergroup $Q(n)$, a more general unimodular Lie supergroup $G$ obtained as the supergroup of invertible elements of an associative superalgebra $A$ with an invariant, non-degenerate odd trace $\otr$. Similarly to $\mathfrak{q}(n)$, let us denote by $\{e_{i}\}$ a basis of $A$ and by $\{\phi^i\}$ the dual basis of $A^*$. Let us also introduce the structure constants $c_{ij}^k$ by $e_i e_j = c_{ij}^k e_k$, cyclically-symmetric coefficients $t_{{i_1}\dots {i_n}} = \otr{e_{i_1}\dots e_{i_n}}$ and the inverse of the pairing $t = (-1)^\hdeg{e_i} t^{ij}e_i \ot e_j \in A\ot A$ with $t^{ij}t_{jk} = \delta^i_k$. In our conventions for supegroups, we have for the coproduct $\square \phi^i = (-1)^{\hdeg{\phi^j}\hdeg{\phi^k}} c^i_{jk} \phi^j \ot \phi^k$ and for the left action of $A$, seen as a Lie algebra of $G$, $(e_a)^\mathrm L \phi^i = (-1)^{\hdeg{e_a}} c^i_{ja}\phi^j$.  
	\medskip
	
	Both sides of Equation \eqref{eq:PhiOddMorph}, when applied to $\gamma_1 \wedge \dots \wedge \gamma_k$, are a sum over all (possibly self-) intersections of loops; we will prove the equality \eqref{eq:PhiOddMorph} term-by-term. Moreover, we can permute both sides such that the BV operators act on the first two loops for the case of an intersection, or the first loop in the case of a self-intersection. Let us treat these cases separately.
	\medskip
	
\begin{figure}[h]
	\centering
	\begin{subfigure}[b]{0.35\textwidth}
		\begin{center}
			\includesvg{images/labelledgoldmanchord.svg}
		\end{center}
		\caption{A chord at an intersection of two loops}
		\label{fig:ChordGoldman}
	\end{subfigure}
	\qquad  
	\begin{subfigure}[b]{0.35\textwidth}
		\begin{center}
			\includesvg{images/goldmanchordfun.svg}
		\end{center}
		\caption{The corresponding function on the moduli space}
		\label{fig:ChordGoldmanFun}
	\end{subfigure}
	\caption{The term of the LHS of \eqref{eq:PhiOddMorph} corresponding to an intersection of two loops.}\label{fig:ChordGoldmanBoth}
\end{figure}

\begin{figure}[h]
	\centering
	\begin{subfigure}[b]{0.35\textwidth}
		\begin{center}
			\includesvg{images/labelledgoldman.svg}
		\end{center}
		\caption{A resolution of intersection from the Goldman bracket}
		\label{fig:BracketGoldman}
	\end{subfigure}
	\qquad  
	\begin{subfigure}[b]{0.35\textwidth}
		\begin{center}
			\includesvg{images/goldmanbracketfun.svg}
		\end{center}
		\caption{The corresponding function on the moduli space}
		\label{fig:BracketGoldmanFun}
	\end{subfigure}
	\caption{The term of the RHS of \eqref{eq:PhiOddMorph} corresponding to an intersection of two loops.}\label{fig:BracketGoldmanBoth}
\end{figure}

	\textbf{intersection of two loops:} Using Theorem \ref{thm:geometricBV}, we get on the LHS of \eqref{eq:PhiOddMorph} the chord diagram as shown on Figure \ref{fig:ChordGoldman}. Using Definition \ref{def:chord}, this term is equal to the function on Figure \ref{fig:ChordGoldmanFun}. The RHS of \eqref{eq:PhiOddMorph} is given by the Goldman bracket from Definition \ref{def:GoldmanTuraev}, i.e. the holonomy of the loop on Figure \ref{fig:BracketGoldman}. This loop is (up to cyclic permutation) equal to $\gamma'_1\gamma''_1\gamma'_2\gamma''_2$ which gives the function on Figure \ref{fig:BracketGoldmanFun}. 
	
	Our goal is to prove the equality of the two function on Figures \ref{fig:ChordGoldmanFun} and \ref{fig:BracketGoldmanFun}. Let us consider the parts of the diagrams below the box marked $\text{hol}'$, which can be seen as odd elements of $(A^*)^{\ot 4}\subset \mathcal O(G)^{\ot 4}$.  Concretely, from Figure \ref{fig:ChordGoldmanFun} we get
	\[ (-1)^{\hdeg{e_a} +\hdeg{e_b} }t^{ab}t_{ij}t_{kl} (e_a)^\mathrm L  \phi^i \otimes \phi^j \otimes  (e_b)^\mathrm L \phi^k \otimes \phi^l, \]
	while from Figure \ref{fig:BracketGoldmanFun} we get
	\[ t_{jilk} \phi^i \otimes \phi^j \otimes \phi^k \otimes \phi^l. \]
	The equality of these tensors can be proven directly using the invariance of the odd trace. Alternatively, we can see both sides as maps $A^{\otimes 4}\to \Pi^{\ot 3} \cong \Pi$, and prove the identity diagramatically (taking care with signs), getting 
	\[ \raisebox{-0.42\height}{\includesvg{images/C1comp.svg}}. \]
 	To obtain the left-most diagram, we use the fact that acting by $(e_a)^\mathrm L$ on a linear function corresponds to right-multiplication by $e_a$. Then, the first equality follows from the invariance of the odd trace, while the second equality follows from cancellation of the pairing $e_i\ot e_j \mapsto t_{ij} = \otr(e_i e_j)$ and $t = (-1)^\hdeg{e_i} t^{ij}e_i \ot e_j$.

\begin{figure}[h]
	\centering
	\begin{subfigure}[b]{0.35\textwidth}
		\begin{center}
			\includesvg{images/labelleturaevchord.svg}
		\end{center}
		\caption{A chord at a self-intersection.}
		\label{fig:ChordTuraev}
	\end{subfigure}
	\qquad  
	\begin{subfigure}[b]{0.35\textwidth}
		\begin{center}
			\includesvg{images/turaevchordfun.svg}
		\end{center}
		\caption{The corresponding function on the moduli space, from Definition \ref{def:deltaintersections}.}
		\label{fig:ChordTuraevFun}
	\end{subfigure}
	\caption{The term of the LHS of \eqref{eq:PhiOddMorph} corresponding to a self-intersection.}\label{fig:ChordTuraevBoth}
\end{figure}

\begin{figure}[h]
	\centering
	\begin{subfigure}[b]{0.35\textwidth}
		\begin{equation*}
		2\bigg( \raisebox{-0.4\height}{\includesvg{images/labelledturaev.svg}} \bigg)
		\end{equation*}
		\caption{A resolution of intersection from the Turaev cobracket. The factor $2$ is introduced in \eqref{eq:PhiOddMorph}. The loop on the right is the first one in the wedge product.}
		\label{fig:CobracketTuraev}
	\end{subfigure}
	\qquad  
	\begin{subfigure}[b]{0.35\textwidth}
		\begin{center}
			\includesvg{images/turaevcobracketfun.svg}
		\end{center}
		\caption{The corresponding function on the moduli space.}
		\label{fig:CobracketTuraevFun}
	\end{subfigure}
	\caption{The term of the RHS of \eqref{eq:PhiOddMorph} corresponding to a self-intersection.}\label{fig:CobracketTuraevBoth}
\end{figure}

	\textbf{self-intersection:} Here, the situation is analogous, with the loops and corresponding functions shown on Figures \ref{fig:ChordTuraevBoth} and \ref{fig:CobracketTuraevBoth}. These give the following elements of $(A^*)^{\ot 3}$: 
	\[ (-1)^{\hdeg {e_a} + \hdeg{e_b}\hdeg{\phi^i} + \hdeg{\phi^i}\hdeg{\phi^j}} t^{ab} t_{ijk} (e_a)^\mathrm L \phi^i \ot (e_b)^\mathrm L \phi^j \ot \phi^k \]
	and
	\[ (-1)^{\hdeg{\phi^k}} t_{ik} t_j \phi^i \ot \phi^j \ot \phi^k \]
	respectively. Again, one can proceed in coordinates, or diagramatically:
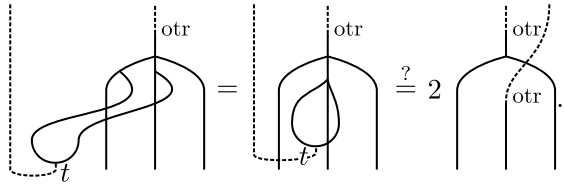
	\[\raisebox{-0.42\height}{\includesvg{images/C2comp.svg}}.\] 
	The last equality does not hold in general, but for $G = Q(n)$ follows from the identity
	\begin{equation}\label{eq:discardyforqn}
	\raisebox{-0.42\height}{\includesvg{images/discardyqn.svg}},
	\end{equation}
	which is proven in Proposition \ref{prop:queerisdiscardy}.
\end{proof}


\begin{proposition}\label{prop:queerisdiscardy}
In the algebra $q_{\mathrm{as}}(n)$, the identity \eqref{eq:discardyforqn} holds.
\end{proposition}
\begin{proof}
	In the notation of the proof of Theorem \ref{thm:holonomyisBVmap}, the identity reads
	\[ -\frac 12 (-1)^{\hdeg {e_a} + \hdeg{e_b}\hdeg{e_i}} t^{ab} c^j_{aib} = t_i u^j, \]
	where $u^je_j\in q_{\mathrm{as}}(n)$ is the unit of the algebra and $e_a e_i e_b= c^j_{aib} e_j$ are the structure constants of the iterated product.
 	\medskip
 	
	For $q_{\mathrm{as}}(1)$, we have $t  =  1\otimes \xi - \xi \otimes 1$,
	and the identity holds since
	\begin{align*}
		-\frac 12 (-1\cdot \xi \cdot \xi - (-1)^{1\cdot 1 } \xi \cdot \xi \cdot 1) = 1 
		&\overset{\checkmark}{=} \mathrm{otr}(\xi) 1, \\
		-\frac 12 (1\cdot 1 \cdot \xi - \xi \cdot 1 \cdot 1) = 0 & \overset{\checkmark}{=} \mathrm{otr}(1) 1.
	\end{align*}
	 The algebra $q_{\mathrm{as}}(n)$ is obtained by tensoring $q_{\mathrm{as}}(1)$ with the algebra $\operatorname{Mat}_n(\mathbb R)$,
	  which satisfies $\sum_{ab}S^{ab}E_{a} X E_{b} = \tr(X) 1_{n\times n}$
	 where $E_{a}$ is a basis of $\operatorname{Mat}_n(\mathbb R)$ and $S^{ab}$ the inverse
	 to $\tr(E_{a}E_{b})$. This is true because a basis $E_{a}$ consists of elementary 
	 matrices, with $a = (\alpha \beta)$ being a pair of indices. 
	 Then the matrix product $S^{ab}E_{a} X E_{b}$ is a matrix with $X_{\beta\beta}$ on
	 the position $(\alpha, \alpha)$ and zeros elsewhere, and summing over all $\alpha${}
	 and $\beta$ gives the identity matrix times the trace of $X$.
	 
	 The tensor product of two such algebras again satisfies \eqref{eq:discardyforqn},
	 which is immediate diagramatically
	 \begin{equation*}
		\raisebox{-0.42\height}{\includesvg{images/doublediscardy.svg}}.
	 \end{equation*}
	 Here $q_{\mathrm{as}}(1)$ is represented by the solid line and  $\operatorname{Mat}_n(\mathbb R)$ by the dash-dotted line.
\end{proof}
\begin{remark} \label{rmk:DiscardyUnimodular}
	Given an associative superalgebra where the identity \eqref{eq:discardyforqn} holds, its commutator Lie superalgebra is unimodular. This can be seen by precomposing the identity \eqref{eq:discardyforqn} with the unit.
\end{remark}

\subsection[Odd determinants and H1(Sigma)]{Odd determinants and $H_1(\Sigma)$}
Let us extend the Lie bialgebra structure from $\Goldmanred{\Sigma}$ to $\Goldmanred{\Sigma}\oplus H_1(\Sigma, \mathbb R)$, 
as in Section \ref{ssec:Goldman}. 
\begin{definition}
	On $\Goldmanred{\Sigma}\oplus H_1(\Sigma)$, define a Lie bracket $[\cdot, \cdot]_\textnormal{G}$ as in Definition \ref{def:Goldmandet}, with $\bigcirc=0$. 
	The cobracket $\delta_\textnormal{T}$ is extended from $\delta_{\mathrm T}$ by $0$ on $H_1(\Sigma)$.
\end{definition}
It is easy to check that one obtains again an involutive Lie bialgebra, and thus $\wedge \left( \Goldmanred{\Sigma}\oplus H_1(\Sigma, \mathbb R) \right)$ 
becomes a BV algebra using \eqref{eq:CEBV}.

The role of the function $\log \det$ from Section \ref{ssec:Goldman} will be played by the following function:
\begin{definition} \label{def:odet}
	The odd determinant $\odet \colon Q(n) \to \Pi$ is the odd function
	on $Q(n)$ defined by 
	\begin{equation}\label{eq:odetdef}\odet (X+ \xi Y) = \sum_{j\ge 1 \; \mathrm{odd}} \frac{\tr( (X^{-1}Y)^{j} )}{j}\,.\end{equation}
\end{definition}
The odd determinant satisfies 
\begin{equation}\label{eq:odetadd}
\odet(GH) = \odet(G) + \odet(H),
\end{equation}
and thus is invariant \cite[Theorem~1.8.5]{Leites2011seminar}.
\begin{proposition}
	Define an algebra map $$\ogmex \colon \wedge ( \Goldmanred{\Sigma}\oplus H_1(\Sigma, \mathbb R) ) \to  \mathcal O( \ModSpace{\Sigma, V}{Q(n)})$$
	by extending $\ogm$ via
	\[\ogmex (a) = \hol_{\gamma_a}(\odet), \]
	where $\gamma_a\in\pi_1(\Sigma)$ represents $a\in H_1(\Sigma) \cong \pi_1(\Sigma)^\mathrm{ab}$. Then $\ogmex$ is a map of (quasi-)BV algebras, with respect to
	$\Delta$ on $ \mathcal O( \ModSpace{\Sigma, V}{Q(n)})$ and $\Delta^{[\cdot, \cdot]_\textnormal{G}, 2\delta_\textnormal{T}}$ on $\wedge \left( \Goldmanred{\Sigma}\oplus H_1(\Sigma, \mathbb R) \right)$.
\end{proposition}
\begin{proof}
	Let us start by stating the odd analogue\footnote{Equation \eqref{eq:derlogdet} says, in other terms, that the Lie algebra morphism $\operatorname{Tr}\colon \mathfrak{gl}(n) \to \mathbb R$ is the differential of the Lie group morphism $\log\det\colon GL_n(\mathbb R)^+ \to \mathbb R$. Similarly, \eqref{eq:odetder} says that the Lie superalgebra morphism $\otr \colon \mathfrak q(n)\to \Pi$ is the differential of the Lie supergroup morphism $\odet \colon Q(n) \to \Pi$.} of Equation \eqref{eq:derlogdet}: an element $x\in \mathfrak{q}(n)\cong T_e Q(n)$, i.e. a derivative at the group identity $e\in Q(n)$, satisfies
	\begin{equation}\label{eq:odetder} x(\odet) = -\otr x. \end{equation}
	This is easily seen from Definition \ref{def:odet}: only the term $j=1$ of the sum \eqref{eq:odetdef} is linear in the odd coordinate, and can have a non-zero contribution when evaluated at $e\in Q(n)$. The sign comes from our convention for the isomorphism $\mathfrak{q}(n)\cong T_e Q(n)$; we identify $x\in \mathfrak q(n)$ with the derivation at $e$ given as $\phi^i \mapsto (-1)^{\hdeg{\phi^i}} \phi^{i}(x)$.
	\medskip
	
	Equations \eqref{eq:odetadd} and \eqref{eq:odetder} imply, for the left-invariant action of $x\in \mathfrak{q}(n)$
	\begin{equation}\label{eq:odetLinv} x^\mathrm{L} \odet  =- \otr x,  \end{equation}
	i.e. a constant function on $Q(n)$. Similarly, again using \eqref{eq:odetadd}, we get
	\[  \raisebox{-0.42\height}{\includesvg{images/odetcoprL.svg}}. \]
	
	Now, we can prove that the extended map $\ogmex$ is a map of quasi-BV algebras. There are three new cases to consider, containing loops $\gamma_a$ for $a\in  H_1(\Sigma, \mathbb R)$
	
	\textbf{self-intersection of a loop $\gamma_a$:} The cobracket is extended by zero to the first homology. Similarly, a chord acting on a triple coproduct of $\odet$ is zero, since at least one leg of the chord will differentiate the constant function.
	
	\textbf{an intersection of two loops $\gamma_a, \gamma_b$:}. The Goldman bracket is extended to $H_1(\Sigma, \mathbb R)$ by zero. On the other hand, the chord acting on two functions $\odet$ will give a function proportional to $(-1)^{\hdeg{e_i}}t^{ij} \otr(e_i) \otr(e_j)$. Since either $e_i$ or $e_j$ are even, their odd trace is $0$ and the function corresponding to the chord also vanishes.
	
	\textbf{an intersection of $\gamma_a$ with $\gamma$:} This is the only nonzero term, let us analyze the two terms similarly as in the proof of Theorem \ref{thm:holonomyisBVmap}.
	The chord and the corresponding function is shown on Figure \ref{fig:ChordExt}. The extended Goldman bracket $[\gamma, a]$ contributes just $\gamma$, for a positive intersection. Using \eqref{eq:odetLinv} in Figure \ref{fig:ChordGoldmanExtFun}, we get the following element of $\mathcal{O}(Q(n))^{\otimes 4}$
	\begin{equation} \label{eq:ExtProof1} (-1)^{\hdeg{e_a}+\hdeg{e_b}+1} t^{ab} t_{ij}  (e_a)^\mathrm{L} \phi^i \otimes \phi^j \otimes \otr(e_b) \otimes 1, \end{equation}
	where $1\in \mathcal O(Q(n))$ is the constant function equal to $1$. If we choose a basis\footnote{Recall that $E_{(\alpha\beta)}$ are the elementary matrices.} of $q_\textnormal{as}(n)$  as $E_{(\alpha\beta)}$ and $\xi E_{(\alpha\beta)}$, the element $t$ can be written as
	\[  t = \sum_{\alpha, \beta =1 }^{n} E_{(\alpha\beta)} \otimes \xi E_{(\beta\alpha)} -  \xi E_{(\alpha\beta)} \otimes  E_{(\beta\alpha)}.\]
	Then, in Equation \eqref{eq:ExtProof1}, only the first term of the above sum contributes, and only when $\alpha = \beta$, i.e. we get the left-invariant action of the identity matrix  $\sum_{\alpha}E_{(\alpha\alpha)}$, which acts trivially
	\begin{equation*}  t_{ij}  (\textstyle\sum_{\alpha}E_{(\alpha\alpha)})^\mathrm{L} \phi^i \otimes \phi^j \otimes 1 \otimes 1 =  t_{ij}  \phi^i \otimes \phi^j \otimes 1 \otimes 1. \end{equation*}
\begin{figure}[h]
	\centering
	\begin{subfigure}[b]{0.35\textwidth}
		\begin{center}
			\includesvg{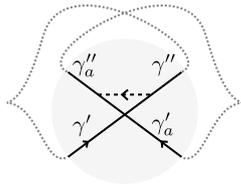}
		\end{center}
		\caption{A chord at an intersection of two loops $\gamma$ and $\gamma_a$}
		\label{fig:ChordGoldmanExt}
	\end{subfigure}
	\qquad  
	\begin{subfigure}[b]{0.35\textwidth}
		\begin{center}
			\includesvg{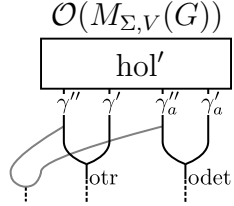}
		\end{center}
		\caption{The corresponding function on the moduli space}
		\label{fig:ChordGoldmanExtFun}
	\end{subfigure}
	\caption{The term corresponding to an intersection of two loops $\gamma$ and $\gamma_a$.}\label{fig:ChordExt}
\end{figure}

\end{proof}

\subsection[Surjectivity of the map Phi]{Surjectivity of the map $\ogmex$}
In this section, we investigate which (algebraic) functions on the moduli space without marked points are in the image of the maps $\ogm$ and $\ogmex$. 

Let us briefly recall the even case. The even analogue of $\ogm$ is the extension of $\egm$ (defined in Theorem \ref{thm:Goldman}) to an algebra map $\Sym \Goldman{\Sigma} \to \mathcal O(\ModSpace{\Sigma}{GL(n)})$. The image of this extension is generated by traces of arbitrary holonomies. Choosing a set of generators of $\pi_1(\Sigma)$ with holonomies denoted by $A_1, \dots, A_N$, this image is generated by traces of monomials in $A^{\pm1}_1, \dots, A^{\pm 1}_N$.

By \cite{Procesi1976}, any function on the space $\operatorname{Mat}_n(\mathbb R)^{\times N}$ polynomial in the matrix entries and invariant under simultaneous conjugation by $GL(n)$ is a product of traces of monomials in the matrices. 

If we restrict to the space $GL(n)^{\times N}/GL(n)$, the algebraic functions we consider are polynomials in the matrix entries and the inverses of the determinants of $A_i$, invariant under simultaneous conjugation. This algebra is equal to the image of the extension of $\egm$: this follows from the fact that the additional generators $\det A_i^{-1}$ can be written as a polynomial in traces of powers of $A_i^{-1}$.
\smallskip

Let us now turn to the odd case. The image of $\ogm$ is generated by odd traces of arbitrary holonomies, while the image of $\ogmex$ has additional generators for odd determinants of holonomies. 

A natural class of algebraic functions on the moduli space $\ModSpace{\Sigma}{Q(n)}\cong Q(n)^{\times N}/Q(n)$ is given by invariant polynomials of matrix entries and inverses of determinants of their even parts.
It was proven by Berele \cite{Berele2013} that all functions on $q_\mathrm{as}(n)^{\times N}$ polynomial in entries and invariant under simultaneous conjugation by $Q(n)$ are products of odd traces of monomials of matrices\footnote{Invariants under conjugation by $Q(n)$ were also studied by Sergeev \cite{Sergeev1985, Sergeev2001} and Shander \cite{Shander1998}}. However, restricting to invertible matrices, already for $N=n=1$ we see that not all algebraic functions are coming from odd traces. Indeed, denoting by $a$ and $da$ even and the odd coordinate $$A = a + \xi da\in Q(1),$$ we get $\otr A^k = k a^{k-1} da$, and we need to include the odd determinant $\odet A = a^{-1} da$, which is also invariant. Thus, we are led to the following conjecture.



\begin{conjecture}\label{conj:inv}
	The algebra of all algebraic functions on $$\ModSpace{\Sigma}{Q(n)} \cong (Q(n)^{\times N})/Q(n)$$ is generated by odd traces of products of matrices and their inverses, and by the odd determinants of the $N$ matrices. 
\end{conjecture}
\begin{proposition}
	For $n=1$, the conjecture is true.
\end{proposition}
\begin{proof}
	On $Q(1)^{\times N}$, we denote the even coordinates $a_{i}$ and the odd
	coordinates $da_{i}$. 	It is enough to look at invariants under $q(1)$, since odd traces and determinants
	are already $Q(1)$ invariant. The even part of $q(1)$ acts trivially, the odd part acts via
	\[ [ \xi, a + \xi da ]  = 2 da\,.\]
	In other words, on $Q(1)^{\times N}$, $\xi$ acts as the odd vector field
	$$ 2 \sum_{i} da_{i} \frac{\partial}{\partial a_{i}}. $$
	Invariant functions on $Q(1)^{\times N}$ are thus the same thing as
	closed differential forms on the complement of coordinate hyperplanes of
	$\mathbb R^{N}$, polynomial in the coordinates and their inverses.
	 We will freely go between function on $Q(1)^{\times N}$ and such differential forms.
	
	The odd trace of the $i$th matrix $A_{i} = a_i + \xi da_i$ 
	is equal to $da_{i}$, its odd determinant equals $da_{i}/a_{i}$. More generally, let $f$ be a non-commutative
	polynomial in $N$ variables, then 
	\[ f(A_{1}, \dots, A_{N}) = \alpha + \xi d\alpha \]
	for some function $\alpha = f(a_{1}, \dots, a_{N}) + \text{forms of degree $\ge 2$}$.
	This is true because it holds for the matrices $A_{i}$ and remains true for product of
	such matrices of that form
	\[  ( \alpha_1 + \xi d\alpha_1)( \alpha_2 + \xi d\alpha_2)
	=( \alpha_{1}\alpha_{2} +
	 d\alpha_{1} d\alpha_{2}) + \xi ( \alpha_{1}d \alpha_{2} + d\alpha_{1} \alpha_{2} ).\]
	
	The cohomology of the algebra of non-constant algebraic functions on $Q(1)$ with respect to
	$d$ is one-dimensional, generated by $da/a$. Thus, any $q(1)$ invariant function on 
	$Q(1)^{\times N}$ can be written as 
	a product of functions $da_{i}/ a_{i}$ plus an exact term. 
	The term $da_{i}/ a_{i}$ is equal to the product of odd determinants,
	we thus need to treat exact forms.
	
	Let $\beta = \beta^{(1)} + \beta^{(2)} + \dots $ be an exact form with $\beta^{(i)}$ an $i$-form. We can always 
	 write $\beta^{(k)} = d (\sum_{|I| = k-1}\beta_{I} da_{I}) = \sum_{|I| = k-1} d (\beta_{I}) da_{I} $ for
	some functions $\beta_{I}$, $I = (I_{1}, \dots, I_{k-1})$ is a multi-index of length $k-1$. 
	Then,  if $\beta^{(i)}$ vanish for $i < k$, the $k$-th form component of 
	\[  \beta - \mathrm{otr}(\beta_{I}(A)) \otr(A^{I_{1}}) \dots \otr(A^{I_{k-1}})  \]
	vanishes. Here $\mathrm{otr}(\beta_{I}(A))$ is obtained by replacing
	$a_{i}$ by $A_{i}$, whose 1-form part does not depend on the chosen order.
	
	Working order by order, we finally write any invariant function as
	a product of odd determinants plus sum of products of odd traces.
\end{proof}
\begin{remark}
In the even case, the kernels of the maps $\egm$ are non-empty for every $n$; however their intersection over $n\in \mathbb N$ is empty, as shown by Etingof \cite{Etingof2006}. It would be interesting to study an analogous question for $Q(n)$.
\end{remark}


%
\newpage


\appendix

\section{Skeletons and foliations}
\subsection{Skeletons}
\label{app:skeleton}
We now prove Proposition \ref{prop:2Dtopology}, saying that all skeletons can be related by a finite sequence of slides. This will follow from the work of Penner, namely from the path-connectedness of the complex of all fatgraphs of a bordered surface see \cite[Ch.~4]{PennerBook}. The additional work in the proof is needed to relate skeletons to uni-trivalend fatgraphs from \cite{PennerBook}.

\begin{proof}[Proof of Proposition \ref{prop:2Dtopology}]
	We start by replacing the skeleton by a 
	 homotopy equivalent graph embedded in $\Sigma$. Specifically, we arbitrarily resolve each marked point $p$, with valence $n(p)$, 
	to a binary tree with a root at $p$ and $n(p)$ leaves. This way, the skeleton $\Gamma$ is transformed into a uni-trivalent
	graph $\tilde{\Gamma}$ in $\Sigma$, dual to a triangulation. 
	In general, let us consider a uni-trivalent graph $\tilde{\Gamma}$ in the surface such that 
	\begin{enumerate}
		\item the univalent vertices are mapped bijectively to $V$,
		\item the trivalent vertices are mapped to the interior of $\Sigma$, and
		\item $\Sigma$ deformation retracts to $\tilde{\Gamma}$.
	\end{enumerate}
	To go back from such uni-trivalent graph to a (possibly different) skeleton, we need to choose a subset $T$ of
	edges of $\tilde{\Gamma}$, let's call them \emph{red}, such that their complement $\tilde{\Gamma} \setminus T$
	is a spanning forest\footnote{Recall that a spanning forest of a graph is a subgraph consisting of a disjoint
	union of trees that contains all vertices of the graph.}, with one tree rooted at each marked point. Then, contracting the 
	trees to their roots, the red edges become a skeleton of the surface.
	
	We will now proceed in two steps. First, we will prove that for a fixed uni-trivalent graph, any two choices of red
	edges give skeletons that can be related by slides. Then, we will use a result of Penner \cite{PennerBook} relating 
	different uni-trivalent graphs via flips.
	\medskip
	
	With two choices $T, U$ of the sets of red edges on the same uni-trivalent graph $\tilde{\Gamma}$, let us denote
	the two complementary forests by $(T_{p_{1}}, \dots, T_{p_{n}})$ and $(U_{p_{1}}, \dots, U_{p_{n}})$,
	where we enumerated their trees by their roots.
	One of the red edges $e_0$ in $T$ must be in $U_{p}$ for some $p\in V$ (otherwise $T = U$).
	Let us consider the union of $e_0$ with the forest $\tilde{\Gamma} \setminus T$. This subgraph will not
	be a forest anymore, but this additional edge\footnote{Spanning tree is equivalently characterized by
	the property that addition of any edge creates a single loop \cite[Sec.~4.4,~Ex.~2.]{MatNes}. For spanning forests,
	the result we use follows by joining all the roots of the trees to a new vertex, creating
	a spanning tree of this enlarged graph. Then, adding an edge 
	creates a loop in this spanning tree, which might or might not pass through the new vertex,
	giving the two possibilities.}
	\begin{enumerate}
		\item either created a loop $\gamma$ in a tree $T_{p_{i}}$, or
		\item it connected $T_{p_{i}}$ with some other tree $T_{p_{j}}$,
		creating a path between $\gamma$ between $p_{i}$ and $p_{j}$.
	\end{enumerate}
	In both cases, the path $\gamma$ has to contain an edge $e_1$ that belongs to $U$. This implies 
	that $e_1 \neq e_0$, as $e_0 \in U_{p} \subset \tilde{\Gamma} \setminus U$. Thus, considering $T' := T \setminus \{e_0\} \cup \{ e_1\}$,
	we obtain a new red set, i.e. a set of edges such that $\tilde{\Gamma} \setminus T'$ is a forest rooted at $V$.
	
	Repeating this move, we transform $T$ to $U$, since each such move
	subtracts 2 from the finite number  $\sum_{i} |T_{p	_{i}} \mathrel{\triangle}  U_{p_{i}}|$.
	\medskip
	
	For the skeletons corresponding to the forests, this move corresponds
	to the slide, along $e_0$, of the red half-edges between the removed red
	 edge $e_0$ and the added red edge $e_1$, as encountered along $\gamma$. On Figure \ref{fig:MoveToSlide}, 
 	we show the case when $\gamma$ is a loop. 	The case when the added edge connects two trees is analogous.

\begin{figure}[h]
	\centering
	\includesvg{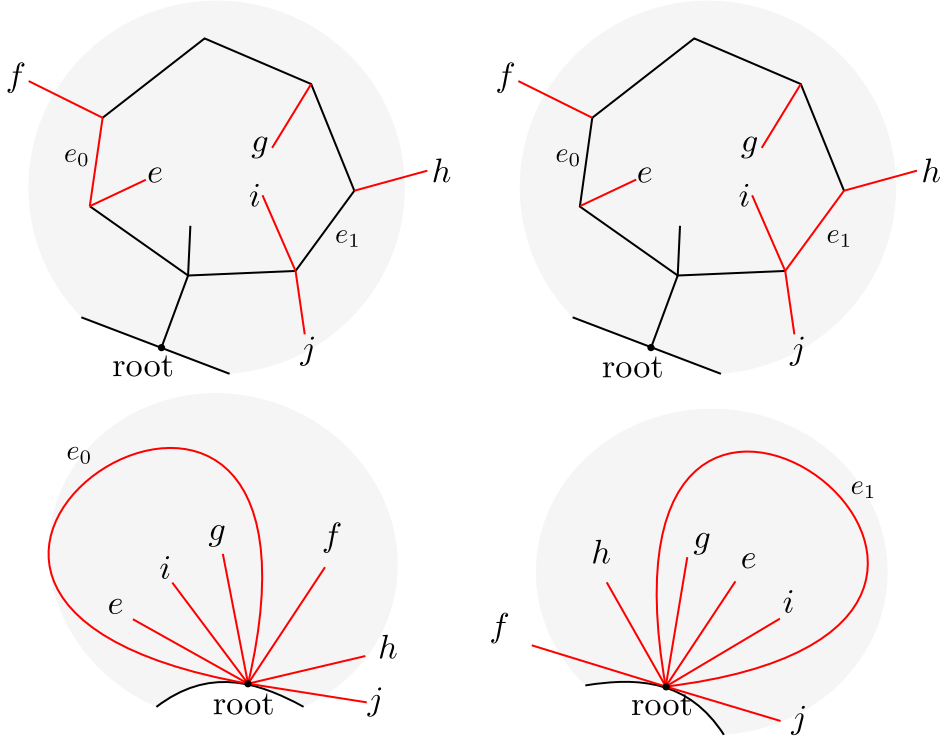}
	\caption{Replacing the red edge $e_0$ with $e_1$. On the top, the red edges and the black forest are shown. On the bottom, the forest is contracted to show the corresponding skeletons. The half-edges $g, f$ and $h$ undergo a slide along $e_0$, because they are between $e_0$ and $e_1$.}
	\label{fig:MoveToSlide}
\end{figure}	

	Now, let us turn to relating different uni-trivalent graphs. We will use the fact that any two uni-trivalent graphs as above are related by a 
	sequence of flips, i.e. moves as on Figure \ref{fig:flip}.
	
\begin{figure}[h]
	\centering
	\includesvg{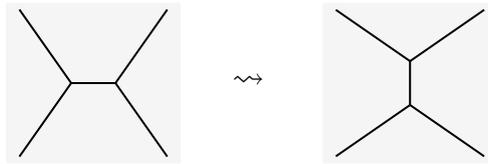}
	\caption{Flip move. The edge in the middle has to connect two distinct trivalent vertices.}
	\label{fig:flip}
\end{figure}

	For surfaces with marked points on the boundary, this claim follows from the work of
	Penner \cite{PennerBook}.   Namely, we use Theorem 5.21 of Chapter 4 in \cite{PennerBook} with $Q$ to be all the marked points on the boundary and $P$ to be all the boundary components without any marked points. As there are no punctured vertices (see Theorem 5.6 in loc.cit.) and thus no quasi-flips, it follows that all uni-trivalent graphs on $\Sigma$ are connected by a sequence of flips.

	\medskip
	
	Finally, let us explain how the flips interact with the skeletons.
	For a flip along an edge, we can find a spanning forest containing such
	edge, by adding this edge to the forest and removing another edge
	from a newly formed loop or path connecting marked points. If this new
	loop or path contains only the one edge, this edge is not valid for flips.
	
	However, flip along an edge contained in the spanning forest 
	has no effect on the corresponding skeleton.  Thus, we can relate
	any two skeletons, replaced by uni-trivalent graphs with red edges, by 
	moving red edges and flips along non-red edges. Since
	moving of red edges corresponds to slides and flips along non-red edges don't
	change the skeleton, the proposition is proven. 
\end{proof}

\subsection{Foliations}
\label{app:foliation}
An example of a foliated surface, with multiple paths and their rotation numbers, is shown below in Figure \ref{fig:foliationexample}
\begin{figure}[h]
	\centering
	\captionsetup{width=0.7\linewidth}
	\includesvg{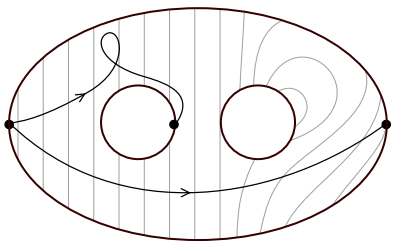}

	\caption{A foliation of a sphere with three punctures. The top path has rotation number $-\frac 12$, the bottom path $0$.}
		\label{fig:foliationexample}
\end{figure}
Now, we classify foliations on surfaces with boundary and marked points.
\begin{proposition}
	If $\Gamma$ is a skeleton of $\Sigma$, then a homotopy class of foliations is uniquely specified by choosing the rotation number for each edge of $\Gamma$. 
	Conversely, for any choice of half-integers for edges of $\Gamma$, there exists a foliation, with these rotation numbers.
\end{proposition}
\begin{proof}
	\textbf{existence:}     For each vertex of the skeleton $\Gamma$, with valence $n$, take the following coupon:
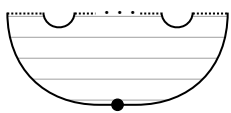
	\begin{center}
		\includesvg{images/coupon.svg}
	\end{center}
    i.e. a disk with $n$ marked (dotted) intervals on the boundary, foliated as above.
    
	For any $r \in \frac 12 \mathbb Z$, construct a strip with $r$ rotations of the foliation by embedding it in a horizontally-foliated plane as a spiral; the case of $r = -3/2$ is shown below:
	\begin{center}
	\includesvg{images/strip.svg}
	\end{center}
	Finally, glue together these strips and the coupons, which is possible as the foliations agree on the dotted intervals.
	
	\textbf{uniqueness:} If we fix one foliation and a metric on $\Sigma$, foliations are in bijection with based maps $\Sigma/V \to S^1$. Homotopy classes of these maps are classified by specifying the degree of the map along each edge of $\Gamma$, as $\Sigma/V \sim \Gamma/V$.
\end{proof}
Let us remark that $1$-dimensional foliations are in 1-1 correspondence with $1$-dimensional distributions in $T\Sigma$,  tangent to $\partial \Sigma$ at $V$. This follows from the Frobenius theorem.  If there exists a nowhere-vanishing section of such distribution, then the corresponding foliation is called \textbf{orientable}. Not all foliations are orientable, i.e. it is not always possible to choose a non-vanishing vector field tangent to the foliation;  the foliation on Figure \ref{fig:foliationexample} is such non-orientable example.

Orientable foliations also arise from framings such that the first component of the framing is tangent to $\partial \Sigma$ at $V$. To each orientable foliation, there is unique-up-to homotopy such framing compatible with the orientation of $\Sigma$.

\section{Lie algebras with a pairing}
\subsection{Ordinary Lie algebras with a pairing}\label{app:evenLie}
Let $\g$ be an ordinary Lie algebra with a symmetric, invariant pairing $\langle-,-\rangle$. For a basis $e_i$ of $\g$, let $s^{ij}$ to be the inverse of the matrix $\langle e_i, e_j \rangle$, and define $f^{ijk} = f^i_{ab}s^{aj}s^{bk}$, where $f^i_{ab}$ are the structure constants of $\g$ defined by $[e_a, e_b] = f^i_{ab} e_i$.  Let $H$ be an ordered set (e.g. the set of half-edges of a vertex of a skeleton, as in Equation \eqref{eq:FRbivector}).

Recall from Section \ref{sec:qBVFRproof} the elements
$$\tilde{s}_{ab} = s^{ij} \iota_{a}(e_{i})\wedge \iota_{b}(e_{j}) \in \wedge^2 \g^{H}$$
and 
$$\tilde{\phi}_{abc} = \frac 1 {24} f^{ijk} \iota_a(e_i) \wedge \iota_b(e_j) \wedge \iota_c(e_k ) \in \wedge^3 \g^{H},$$
where $a, b, c\in H$ and $\iota_a\colon \g \to \g^H$ is the inclusion into the $a$-th factor.
 Let us now collect some useful properties of these two elements of $\bigwedge \g^{H}$.
\begin{proposition}[Properties of $s$]\label{prop:sproperties}
	We have $\tilde{s}_{ab} = - \tilde{s}_{ba}$ and $\tilde{\phi}_{abc}$ is symmetric w.r.t. permutation of its labels. 
	For $a\neq b$
	\[ [\tilde{s}_{ab}, \tilde{s}_{ab}] = 24(\tilde{\phi}_{aab} + \tilde{\phi}_{abb}) \,,\]
	and for $a, b$ and $c$ different, 
	\[ [\tilde{s}_{ab}, \tilde{s}_{bc}] =  -  24 \tilde{\phi}_{abc}. \]
	Finally, for $a$, $b$, $c$ and $d$ all distinct, 
	$ [\tilde{s}_{ab}, \tilde{s}_{cd}] = 0$.
\end{proposition}
\begin{proof}
	The symmetry properties follow from (anti-)symmetry of $s^{ij}$ and $\phi^{ijk}$.
	Next,
	\begin{align*}	[\tilde{s}_{ab}, \tilde{s}_{ab}] &= t^{ij} t^{kl} (
	\iota_{a}([e_{i}, e_{k}])\wedge \iota_{b}(e_{j})\wedge \iota_{b}(e_{l})
	+ \iota_{a}(e_{i})\wedge \iota_{a}(e_{k}) \wedge \iota_{b}([e_{j}, e_{l}]) ) 
	\\&=  f^{mjl} \iota_{a}(e_{m})\wedge \iota_{b}(e_{j})\wedge \iota_{b}(e_{l})
	+ f^{nik} \iota_{a}(e_{i})\wedge \iota_{a}(e_{k}) \wedge \iota_{b}(e_n)
	\\&=24\tilde\phi_{abb} + 24\tilde\phi_{aab}\,.	
	\end{align*}
	For one common index, we have
	\begin{align*}	[\tilde{s}_{ab}, \tilde{s}_{bc}] &= t^{ij} t^{kl} 
	\iota_{a}(e_{i})\wedge \iota_{b}([e_{j}, e_{k}])\wedge \iota_{c}(e_{l}) 
	\\&=  f^{nil}\iota_{a}(e_{i})\wedge \iota_{b}(e_{n})\wedge \iota_{c}(e_{l}) = -24\tilde \phi_{abc}\,.
	\end{align*}
\end{proof}
\subsection{Lie superalgebras with an odd pairing}\label{app:oddLie}
Let us  denote by $e_{i}$ a homogeneous basis of an odd metric Lie algebra $\g$ as in Definition \ref{def:oddinvariantpairing}. Denote $t_{ij} = \langle e_{i}, e_{j} \rangle$ and $[e_{i}, e_{j}] = f^{k}_{ij}e_{k}$.
Then we can express $\phi$ and $\nu$ from Definition \ref{def:oddinvariantpairing} in coordinates and also define an invariant
element in $\bigwedge^{2}\g$ inverse to the pairing on $\g$.

\begin{proposition}\label{prop:oddmetric} \hspace{1mm}
	\begin{enumerate}
		\item	The Cartan element $\phi$ is invariant and graded-symmetric and thus defines an element
		$\phi \in (\mathrm{Sym}^{3}\g)^{\g}$. In coordinates, we have 
		\begin{equation}
		\phi = \phi^{xyz} e_{x} e_{y} e_{z} = \frac{(-1)^{\hdeg{e_{y}}}}{24}  t^{xj}f_{jk}^{y}t^{kz} e_{x}e_{y}e_{z},
		\end{equation}
		where $t^{ij}$ is the matrix inverse to $t_{ij}$
		\item The element $(-1)^{\hdeg{e_{i}}} t^{ij} e_{i}\wedge e_{j} \in \bigwedge^{2}\g$ is $\g$-invariant.
		\item The element 
		$(-1)^{\hdeg{e_{i}}} t^{ij} f_{ij}^{k} e_{k}\in \g$ is equal to $\nu$ and is 
		in the center of $\g$.
	\end{enumerate}
\end{proposition}
\begin{proof}
	\begin{enumerate}
		\item From the definition of $\phi$, we have
		\[ \phi(\alpha, \gamma, \beta) = 
		(-1)^{(\hdeg{\gamma}+1)(\hdeg{\beta}+1)+1 +\hdeg\gamma -\hdeg\beta}\phi(\alpha, \beta, \gamma){}
		= (-1)^{\hdeg{\beta}\hdeg{\gamma}}\phi(\alpha, \beta, \gamma).\]
		From the invariance of the pairing we have for any $x, y, z\in\g${}
		\[ \langle x, [y, z] \rangle = -(-1)^{\hdeg{x}\hdeg{y}} \langle [y, x], z \rangle{}
		=  \langle z, [x, y] \rangle.\]
		Thus, $\phi(\alpha, \beta, \gamma) \propto (-1)^{\hdeg{\beta}}\langle t^{\#}\alpha , [t^{\#}\beta, t^{\#}\gamma]\rangle$ 
		satisfies
		\[\phi(\beta, \gamma, \alpha) = (-1)^{\hdeg{\alpha} + \hdeg\beta}\phi(\alpha, \beta, \gamma){}
		= (-1)^{\hdeg{\gamma}}\phi(\alpha, \beta, \gamma),\]
		which means $\phi$ is symmetric (since $\hdeg \alpha + \hdeg \beta + \hdeg \gamma \stackrel{\mathrm{mod}\,2}{=} 0$ for the result to be non-zero).
		
		The invariance follows from this as
		\begin{align*} &\langle [w, x], [y, z]\rangle  + (-1)^{\hdeg w \hdeg x}\langle x, [w, [y, z]] \rangle
		\\&  =
		\langle [w, x], [y, z]\rangle  + (-1)^{\hdeg w \hdeg x}\langle [y, z], [x, w] \rangle\,. \end{align*}
		Now just use that $ \phi (\ad_{x} (\alpha\ot \beta \ot \gamma))$ is (up to a numerical factor)
		equal to $\langle \cdot , [\cdot,\cdot]\rangle$ evaluated on $ \ad_{x} (t^{\#}\alpha \ot t^{\#}\beta \ot t^{\#}\gamma) $
		\medskip{}

		Evaluating $\phi$ on three basis elements of $\g^{*}$, we get
		\begin{align*} &\phi(e^{x}, e^{y}, e^{z}) = (-1)^{\hdeg{e_{y}}\hdeg{e_{x}}}\phi(e^{y}, e^{x}, e^{z}) 
		\\&= (-1)^{\hdeg{e_{x}}+\hdeg{e_{y}}\hdeg{e_{x}}}\frac 1 {24} e^{y} ( [t^{xi}e_{i}, t^{zj} e_{j}] ) 
		= (-1)^{\hdeg{e_{x}}+\hdeg{e_{y}}\hdeg{e_{x}}}\frac 1 {24}  t^{xi} f^{y}_{ij} t^{jz} , \end{align*}
		where we used that $t^{ij} = (-1)^{\hdeg{e_{i}}\hdeg{e_{j}}} t^{ji} = t^{ji}$.
		This corresponds to the formula $\phi =(-1)^{\hdeg{e_{y}}}\frac 1 {24} 
		t^{xj}f_{jk}^{y}t^{kz} e_{x} e_{y} e_{z}$,
		since pairing this with $e^{x}, e^{y}, e^{z}$ gives
		\[ (-1)^{\hdeg{e_{y}} +\hdeg{e_{x}}(\hdeg{e_{y}}+\hdeg{e_{z}}) +\hdeg{e_{y}}\hdeg{e_{z}} }
		\frac 1 {24} t^{xj}f_{jk}^{y}t^{kz} \]
		and these two signs are equal, since $\hdeg{e_{x}} + \hdeg{e_{y}} + \hdeg{e_{z}} = 0$. Note that the 
		sign $(-1)^{\hdeg{e_{x}}(\hdeg{e_{y}}+\hdeg{e_{z}}) +\hdeg{e_{y}}\hdeg{e_{z}}}$ is the Koszul
		sign from the pairing of $e_{x} \ot e_{y} \ot e_{z}$ with $e^{x}\ot e^{y}\ot e^{z}$.
		
		\item It follows from a direct calculation, using the symmetry of $\phi$, that
		$$\ad_{e_{k}} (-1)^{\hdeg{e_{i}}}t^{ij} e_{i} \wedge e_{j} = 
		(-1)^{\hdeg{e_{k}}} 48 \,t_{kc}\phi^{cxy} e_{x} \wedge e_{y} = 0.$$
		Alternatively,
		see Remark \ref{rem:decalage}.
		\item By definition, $\nu$ is equal to $\nu^{i}e_{i} = (-1)^{\hdeg{e_{k}}}f_{lk}^{k}t^{li}e_{i}$. From
		the symmetry of $\phi^{xyz}$, we get $(-1)^{\hdeg{e_{k}}}t^{il}f^{k}_{lx} = (-1)^{\hdeg{e_{k}}}f^{i}_{xl}t^{lk}$, 
		which gives
		\[  (-1)^{\hdeg{e_{k}}}t^{il}f_{lk}^{k} =  (-1)^{\hdeg{e_{k}}}f_{kl}^{i}t^{lk} \,.\]
		Because the Lie bracket commutes with the action of $\g$, we get that $\nu = (-1)^{\hdeg {e_{i}}} t^{ij}[e_{i}, e_{j}]$
		is invariant by the previous point.
	\end{enumerate}
\end{proof}

Recall that $\Pi$ is the vector space $\mathbb R^{0|1}$ and $\Pi \g := \Pi \otimes \g$.
The implicit Koszul sign from commuting with $\Pi$ allows us to state the above result
more invariantly.
\begin{remark} \label{rem:decalage}
	One can use the isomorphism $\g^{*} \cong \Pi \g${}
	and the \emph{décalage} isomorphism $\bigwedge^{n}(\Pi\g) \cong \Pi^n\otimes \mathrm{Sym}^{n}(\g)$ to 
	clarify the previous proposition. The map $x\ot y\ot z \mapsto \langle x, [y, z]\rangle${}
	is graded antisymmetric, and thus can be seen as an odd element of $\bigwedge^{3}(\g^{*})$.
	Using the following sequence of $\g$-equivariant isomorphisms, 
	$$ \Pi\otimes\bigwedge^{3}(\g^{*}) \cong \Pi\otimes\bigwedge^{3}(\Pi\g) \cong \mathrm{Sym}^{3}(\g), $$
	this element is mapped to $-\phi$.
	Similarly, the odd pairing lives in $\mathrm{Sym}^{2}(\g^{*})\cong \mathrm{Sym}^{2}(\Pi\g) \cong  \bigwedge^{2}(\g)$; $t_{ij}e^{i}e^{j}$ gets sent to $(-1)^{\hdeg{e_{i}}}t^{ij} e_{i}\wedge e_{j}$, on which we can apply the Lie bracket to get $\nu$.
\end{remark}

Note that we can map $\phi$ to $U\g$ using the symmetrization map. Explicitly,
$\phi$ as an element of $U\g$ is equal to $\phi^{xyz} e_{x} e_{y} e_{z} =
\frac{(-1)^{\hdeg{e_{y}}}}{24}  t^{xj}f_{jk}^{y}t^{kz} e_{x}e_{y}e_{z}$.
\medskip

Let us also prove the odd analogue of Proposition \ref{prop:tproperties} for elements $\tilde \nu, \tilde t$ and $\tilde \phi$, defined in the proof of Theorem \ref{thm:qBVFR}. The only added feature is that $\tilde t_{aa}$ is not zero.
\begin{proposition}[Properties of $t$]\label{prop:tproperties}
	We have $\tilde{t}_{ab} = - \tilde{t}_{ba}$ for $a\neq b$ and
	$\tilde{t}_{aa} = \iota_{a}(\nu)/2$. The element
	$\tilde{\phi}_{abc}$ is symmetric w.r.t. permutation of its labels. Both are
	invariant under the diagonal action of $\g$.
	
	For $a\neq b$
	\[ [\tilde{t}_{ab}, \tilde{t}_{ab}] = 24(\tilde{\phi}_{aab} + \tilde{\phi}_{abb}) \,,\]
	and for $a, b$ and $c$ different, 
	\[ [\tilde{t}_{ab}, \tilde{t}_{bc}] =  -  24 \tilde{\phi}_{abc} .\]
	Finally, for $a$, $b$, $c$ and $d$ all distinct, 
	$ [\tilde{t}_{ab}, \tilde{t}_{cd}] = 0$.
\end{proposition}
\begin{proof}[Proof of Proposition \ref{prop:tproperties}]
	For $\tilde{t}_{aa}$, we get 
	\begin{align*} \tilde{t}_{aa} &= (-1)^{\hdeg {e_{i}}} t^{ij}  \iota_{a}{(e_{i})} \iota_{a}{(e_{j})}
	\\&=  (-1)^{\hdeg {e_{i}}} \tfrac 12 t^{ij} (\iota_{a}{(e_{i})} \iota_{a}{(e_{j})}- \iota_{a}{(e_{j})} \iota_{a}{(e_{i}})) 
	\\&= \tfrac 12 \iota_{a} \left((-1)^{\hdeg {e_{i}}} t^{ij} f_{ij}^{k} e_{k}\right).\end{align*}
	
	Let us calculate only $[\tilde{t}_{ab}, \tilde{t}_{bc}]$:
	\begin{align*}
	[\tilde{t}_{ab}, \tilde{t}_{bc}] &= 
	(-1)^{\hdeg{e_{i}}+\hdeg{e_{k}}} t^{ij} t^{kl} 
	\iota_{a}(e_{i}) \iota_{b}([e_{j}, e_{k}]) \iota_{c} (e_{l})
	\\&= (-1)^{\hdeg{e_{i}}+\hdeg{e_{j}}+\hdeg{{e_{m}}}} t^{ij} f^{m}_{jk} t^{kl} 
	\iota_{a}(e_{i}) \iota_{b}(e_{m}) \iota_{c} (e_{l})
	\\&= - 24\phi^{iml}	\iota_{a}(e_{i}) \iota_{b}(e_{m}) \iota_{c} (e_{l}).
	\end{align*}
	The calculation for $[\tilde{t}_{ab}, \tilde{t}_{ab}]$ is analogous.
\end{proof}

\begin{remark}
	The elements $s/t$ and $\phi$ satisfy similar relations, as shown in Proposition \ref{prop:sproperties} and \ref{prop:tproperties}.
	This can be seen as having a morphism to $\bigwedge \g^{H}$ or $U\g^{H}$
	from a super analogue of the Drinfeld-Kohno Lie algebra, which we will call $\mathfrak{\hat{p}}_{\mathrm{odd}}(n)$, with $n = |H|$.
	
	The algebra $\mathfrak{\hat{p}}_{\mathrm{odd}}(n)$ is generated by odd elements
	$\{t_{ab}\}_{a, b \in \{1, \dots, n\}}$ and even elements $\{\phi_{aaa}\}_{a \in \{1, \dots, n\}}$,
	where $t_{aa}$, $\phi_{aaa}$ are central and $t_{ab}$ satisfy $t_{ba} = - t_{ab}$ for $a\neq b$
	and \[ [t_{ab}, t_{ac} + t_{bc}] = 0, \quad \quad [t_{ab}, t_{cd}] = 0 \]
	for $a, b, c$ or $a, b, c, d$ all distinct. The elements $\phi_{aab}+\phi_{abb}$
	and $\phi_{abc}$ can be defined as the commutators of $t$'s; the relation
	$[t_{ab}, t_{ac} + t_{bc}] = 0$ tells us that $\phi_{abc}$ is symmetric.
	
	A $\mathbb Z$-graded version of this
	algebra, without the central elements, appears in the study of 
	rational cohomology of the little $n$-discs operad for odd $n$, 
	see \cite[Part~II, Section~14.1.1 and Theorem~14.1.14]{Fresse}
\end{remark}

\section{Hopf-like algebras governing the fusion of quasi-BV structures}\label{app:fusion}
We can characterize $\g$-quasi-BV manifolds as being manifolds with action of 
an algebra $\mathcal H^\g$ by differential operators, which we define below. Then, the fusion of quasi-BV manifolds is captured by a coproduct-like structure on these algebras.	
\subsection{Quasi-BV manifolds and Hopf algebras}
\begin{definition}
	Let $(\g, t)$ be an odd metric Lie algebra. Define
	\[\mathcal H^\g \equiv U \g \otimes U (\fld \BV),\]
	where the odd generator $\BV$ (graded) commutes with the Lie algebra $\g$ and satisfies 
	\[\BV^2 \equiv \frac 1{24} f^{ijk} e_i e_j e_k \in \Sym^3(\g) \subset U \g \,.\]
	Here $f^{ijk} =  (-1)^{k+jk} f_{ab}^k t^{ai}  t^{bj}$ is graded symmetric, so one can write the image of $f^{ijk} e_i e_j e_k$ in $U \g$ as just $f^{ijk} e_i e_j e_k$. 
	\medskip

	The coproduct is defined as 
	\begin{align*}
	\tilde\cp (\xi) &= \xi \otimes 1 + 1 \otimes \xi, \;\; \text{where } \xi \in \g \,, \\
	\tilde\cp (\BV) &=  \BV\otimes 1 + 1\otimes \BV +  \frac{1}{2}(-1)^{i}  t^{ij} e_i \otimes e_j \,.
	\end{align*}
	
	The antipode turns out to be the regular antipode for $\xi \in \g$ and
	$$S(\BV) = -\BV + \nu/4.$$
%
	Finally, let us introduce an additional filtration on $\mathcal H^\g$, which is the usual filtration on $U\g$ and where $\BV$ increases filtration degree by 2.
\end{definition}

On a  supermanifold $M$ with \g-action, a quasi-BV structure can be encoded into a morphism from $\mathcal H^\g$ into the algebra of differential operators on $M$:
\begin{proposition}\label{PROPHopf}
	The above definition makes $\mathcal H^\g$ into a Hopf algebra.
	
	A \emph{\g-quasi-BV structure on $M$} is equivalently given by an algebra map $\mathcal H^\g \to \operatorname{DiffOp}(M)$ 
	which preserves parity, filtration and sends $\Delta$ and any $\xi \in \g$ to operators annihilating the constant function $1\in \mathcal O(M)$.
	\end{proposition}
We will denote the image of $\BV$ by $\BV$ again.
\begin{proof}
This means that elements of $\g$ are mapped to vector fields and $\BV$ is mapped to an odd second-order differential operator $\BV$, such
that $M$ has an action of $\g$ and the operator $\BV$ squares to the action of $\phi \in U\g$.
\end{proof}
\subsection{Fusion}

We will now introduce two classes of algebra maps between these Hopf algebras. For moduli spaces of flat connections, these maps will correspond to fusion and disjoint union of the underlying surfaces.
\begin{definition}
	Let $\g, \h, \h_{1,2}$ be odd metric Lie superalgebras. Let $${\cp}\colon \mathcal H^{\g\oplus \h} \to \mathcal H^{\g\oplus \g\oplus \h}$$ be the algebra map defined by 
	sending $\BV$ to $\BV + \tfrac{1}{2} (-1)^{e_i} t_\g^{ij} e_i \ot e_j$, where the second term is an element of $\g\ot \g \subset U(\g)^{\ot2}\cong U(\g\oplus \g)$, and by sending an element $\xi\in \g$ to $(\xi, \xi)\in \g\oplus\g$. 
	
	Let $$i\colon \mathcal H^{\h_1 \oplus \h_2} \to \mathcal H^{\h_1}\otimes \mathcal H^{\h_2}$$ be the algebra map defined by sending  $\BV$ to $\BV\ot1+1\ot \BV$ and by sending $\eta_1\in\h_1$ to $\eta_1\ot 1$ and $\eta_2\in \h_2$ to $1\ot \eta_2$.
	
%
%
%
	
	A fusion at $\g$ of a $\g\oplus\g\oplus \h$-quasi-BV manifold $M$ is defined by the composition
	\[ \mathcal H^{\g\oplus \h} \xrightarrow{\cp} \mathcal H^{\g\oplus\g\oplus \h}\to \operatorname{DiffOp}(M). \]

	A fusion at $\g$ of a $\g\oplus \h_1$-quasi-BV manifold and $\g\oplus \h_2$-quasi-BV manifold is defined by the composition
\[ \mathcal H^{\g\oplus \h_1\oplus \h_2} \xrightarrow{i\circ \cp} \mathcal H^{\g\oplus \h_1}\ot \mathcal H^{\g\oplus \h_2} \to \operatorname{DiffOp}(M)\otimes \operatorname{DiffOp}(N) \to \operatorname{DiffOp}(M\times N). \]
\end{definition}

%

\begin{proposition}\label{PROPFusion}
	The maps ${\cp}$ and $i$ are well defined. The fusion is associative in the sense that the following diagram commutes
\[
\begin{tikzcd}
\mathcal H^{\g\oplus \h }  \arrow[r, "\cp"] \arrow[d, "\cp"] & 	\mathcal H^{\g\oplus\g \oplus \h} \arrow[d, "\cp_2"] \\
\mathcal H^{\g\oplus\g\oplus \h} \arrow[r, "\cp_1"] & 	\mathcal H^{\g\oplus\g\oplus\g\oplus \h}
\end{tikzcd},
\]
where in $\cp_{1,2}\colon \mathcal H^{\g\oplus\g\oplus \h} \to \mathcal H^{\g\oplus\g\oplus\g\oplus \h} $, the index specifies on which $\g$ we act.
\end{proposition}
\begin{proof}
	For the maps $\cp$ and $i$, the only nontrivial identity is $(\cp(\BV))^2 = \cp(\BV^2)$. This is
	a calculation analogous to the one in the proof of Proposition \ref{PROPHopf}.
	
	Both legs of the associativity diagram act as a triple coproduct on $\g$, identity on $\h$ and on $\BV$, they give
	\begin{align*}
	\BV \mapsto &\BV\ot 1 \ot 1 + 1\ot \BV\ot 1 + 1\ot 1 \ot \BV \\&+ \frac 12 (-1)^{\hdeg{e_i}} t_\g^{ij} ( e_i\ot e_j \ot 1 + e_i \ot 1 \ot e_j + 1 \ot e_i \ot e_j  ).
	\end{align*}
\end{proof}
\begin{remark}
	For $\h_1 = \h_2 = 0$, the composition $i\circ \cp$ is the coproduct $\tilde{\cp}$ on $\mathcal H^{\g}$. 	
\end{remark}
Finally, we give a topological interpretation to the maps $\cp$ and $i$.
\begin{proposition}
	Let  $\mathcal{H}^{\g\oplus\g\oplus\h}\to \operatorname{DiffOp}(M_{\Sigma, \{p, p', \dots\}}(G))$ be the quasi-BV structure from Theorem \ref{thm:qBVFR}, with
	the two $\g$ actions corresponding to the two points $p, p'$. Then the $\g\oplus\h$-quasi-BV structure on the surface given by fusion of $p, p'$ into $p''$ is given by the composition 
	\[\mathcal{H}^{\g\oplus\h}\xrightarrow{\cp} \mathcal{H}^{\g\oplus\g\oplus\h} \to \operatorname{DiffOp}(M_{\Sigma, \{p'', \dots\}}(G))\cong \operatorname{DiffOp}(M_{\Sigma, \{p, p', \dots\}}(G)).\]
	
	Similarly, if $\Sigma_1, V_1$ and $\Sigma_2, V_2$ are two surfaces with corresponding $\g^{V_{1,2}}$-quasi-BV structures, then the quasi-BV structure on $(\Sigma_1 \sqcup \Sigma_2, V_1 \sqcup V_2)$ is given by the composition 
	\begin{align*} \mathcal{H}^{\g^{V_1}\oplus\g^{V_2}} &\xrightarrow{i} \mathcal{H}^{\g^{V_1}}\otimes \mathcal{H}^{\g^{V_2}} \\
	&\to \operatorname{DiffOp}(M_{\Sigma_1, V_1}(G)\times M_{\Sigma_2, V_2}(G)) \cong \operatorname{DiffOp}(M_{\Sigma_1\sqcup\Sigma_2, V_1\sqcup V_2}(G)).  \end{align*}
\end{proposition}
\begin{proof}
	The additional term in \eqref{eq:fusionBV} comes from $\cp(\BV) = \BV + \tfrac 12 (-1)^{e_i} t_\g^{ij} e_i \ot e_j$. For the disjoint union of surfaces, the quasi-BV operators are simply added together.
\end{proof}

\printbibliography[heading=bibintoc, title={Bibliography}]

\end{document}